\documentclass[letterpaper,final,twocolumn]{IEEEtran}

\usepackage{graphicx}
\usepackage{subfigure}
\usepackage{amsmath}
\usepackage{amsthm}
\usepackage{amssymb}
\usepackage{color}
\usepackage{xspace}
\usepackage{cite}

\usepackage[colorlinks = true, linkcolor = blue, citecolor =
blue]{hyperref}

\usepackage{algorithmic}
\usepackage{algorithm}

\newtheorem{theorem}{Theorem}[section]
\newtheorem{proposition}[theorem]{Proposition}
\newtheorem{lemma}[theorem]{Lemma}
\newtheorem{corollary}[theorem]{Corollary}
\newtheorem{remark}[theorem]{Remark}
\newtheorem{definition}[theorem]{Definition}

\newcommand{\longthmtitle}[1]{\mbox{}\textit{(#1).}}

\newcommand{\oprocendsymbol}{\hbox{$\bullet$}}
\newcommand{\oprocend}{\relax\ifmmode\else\unskip\hfill\fi\oprocendsymbol}

\newcommand{\real}{\ensuremath{\mathbb{R}}}

\newcommand{\realnegative}{\ensuremath{\mathbb{R}_{<0}}}
\newcommand{\realnonnegative}{\ensuremath{\mathbb{R}_{\ge 0}}}
\newcommand{\integers}{{\mathbb{Z}}}
\newcommand{\integerspositive}{{\mathbb{N}}}
\newcommand{\integersnonnegative}{{\mathbb{N}}_0}

\newcommand{\Ac}{\mathcal{A}}

\newcommand{\Cc}{\mathcal{C}}
\newcommand{\Dc}{\mathcal{D}}

\newcommand{\Fc}{\mathcal{F}}

\newcommand{\Lc}{\mathcal{L}}

\newcommand{\Tc}{\mathcal{T}}

\newcommand{\Vc}{\mathcal{V}}

\newcommand{\xv}{x}
\newcommand{\vv}{v}
\newcommand{\dxv}{\dot{x}}
\newcommand{\dvv}{\dot{v}}
\newcommand{\uv}{u}

\newcommand{\tauocc}{\tau^{\text{occ}}}

\newcommand{\nunom}{\nu^{\text{nom}}}
\newcommand{\Dcnom}{\Dc^{\text{nom}}}
\newcommand{\Tnom}{T^{\text{nom}}}
\newcommand{\Tfol}{T^{\text{fol}}}
\newcommand{\taue}{\tau^e}
\newcommand{\taul}{\tau^l}

\newcommand{\Ta}{T^a}
\newcommand{\Tx}{T^{\text{exit}}}
\newcommand{\Te}{T^e}
\newcommand{\Tl}{T^l}
\newcommand{\ulinevv}{\underline{v}}
\newcommand{\Tiat}{T^{iat}}
\newcommand{\tstop}{t^{\text{stop}}}

\newcommand{\guc}{g_{uc}}
\newcommand{\gsf}{g_{sf}}
\newcommand{\gus}{g_{us}}

\newcommand{\until}[1]{\{1,\dots,{#1}\}}
\newcommand{\ntil}[1]{\{2,\dots,{#1}\}}

\newcommand{\bartauocc}{\bar{\tau}^{\text{occ}}}

\newcommand{\localvehcontrol}{\texttt{\small local vehicular}\xspace}%

\graphicspath{{figures/}}

\begin{document}

\title{Distributed control of vehicle strings \\
  under finite-time and safety specifications\thanks{A preliminary
    version of this work appeared as~\cite{PT-JC:15-necsys} at the 5th
    IFAC Workshop on Distributed Estimation and Control in Networked
    Systems.}}

\author{Pavankumar Tallapragada \qquad Jorge Cort{\'e}s
  \thanks{Pavankumar Tallapragada is with the Department of Electrical
    Engineering, Indian Institute of Science, Bengaluru {\tt\small
      pavant@ee.iisc.ernet.in} and Jorge Cort{\'e}s is with the
    Department of Mechanical and Aerospace Engineering, University of
    California, San Diego {\tt\small cortes@ucsd.edu}}%
}

\maketitle

\begin{abstract}
  This paper studies an optimal control problem for a string of
  vehicles with safety requirements and finite-time specifications on
  the approach time to a target region.  Our problem formulation is
  motivated by scenarios involving autonomous vehicles circulating on
  arterial roads with intelligent management at traffic intersections.
  We propose a provably correct distributed control algorithm that
  ensures that the vehicles satisfy the finite-time specifications
  under speed limits, acceleration saturation, and safety
  requirements. The safety specifications are such that collisions can
  be avoided even in cases of communication failure.  We also discuss
  how the proposed distributed algorithm can be integrated with an
  intelligent intersection manager to provide information about the
  feasible approach times of the vehicle string and a guaranteed bound
  of its time of occupancy of the intersection. Our simulation study
  illustrates the algorithm and its properties regarding approach
  time, occupancy time, and fuel and time cost.
\end{abstract}


\section{Introduction}

In this paper we are motivated by the vision for urban traffic with
coordinated computer-controlled vehicles and networked intersection
managers.  Emerging technologies in autonomy and communication offer
the opportunity to radically redesign our transportation systems,
reducing road accidents and traffic collisions and positively
impacting safety, traveling ease, travel time, and energy consumption.
For example, cruise control and coordination of vehicles (e.g.,
platooning) could ensure smoother (with reduced stop-and-go), safer
and fuel-efficient traffic flow.  This vision involves, among many
other things, scheduling of vehicles' usage of an intersection and the
vehicles optimally meeting those schedules under safety
constraints. Distributed algorithmic solutions are necessary in order
to produce real-time implementations under the computationally-heavy
tasks involved. To this end, here we explore a generalized problem of
distributed control of vehicle strings under specifications of
reaching a target in a fixed finite time while respecting safety
specifications.

\subsubsection*{Literature review}

The control and coordination of multi-vehicle systems in the context
of transportation has a long history starting with the platooning
problem formalized in~\cite{WSL-MA:66}. This classical and active area
of research is so vast that a fair and complete overview is beyond the
scope of this paper. The recent survey
papers~\cite{MJ:14-sv,SEL-YZ-KL-JW:15}, however, provide a good
introduction to the topic of vehicle platoon control and its
literature.  The prototypical aim in vehicle platooning is to
asymptotically achieve a given constant inter-vehicular distance (or a
given constant headway (time) between successive
vehicles~\cite{DS-JKH-CCC-PI:94}) while ensuring all vehicles move at
a desired speed. Hence the topic is often called `string
stability'. The problem is typically formalized as an asymptotic
stabilization problem or an infinite-horizon optimal control problem,
with non-collision constraints that ensure vehicles are separated by
at least a given fixed amount.  With some
exceptions~\cite{RK-PF-JF:11, RK-etal:12, WBD-DSC:12}, constraints on
acceleration control or on state variables, such as speed limits, are
not considered. Recent research on string stability also seeks to
address the challenges that arise due to coordination via wireless
communication channels such as sampling~\cite{SO-JP-NvdW-HN:14},
communication delays~\cite{AAP-RHM-OM:14, SS-CO-SW-AJ:17} and limited
communication range~\cite{RHM-JHB:10}. The work~\cite{BB-MRJ-PM-SP:12}
examines the question of whether local feedback is sufficient to
ensure coherence of large networks under stochastic
disturbances. Given that the string stability problem has largely been
motivated by cruise control of vehicle platoons on highways, it is not
surprising that finite-time constraints on the states of the vehicles
are also usually not considered. This is a major difference with
respect to our treatment in this paper. Since we are motivated by the
problem of coordination of a group of vehicles on arterial roads with
intersections, the consideration of finite-time constraints on the
vehicle string is key.

In the literature on coordination-based intersection management, the
explicit control of platoons has been rarely considered with some
notable exceptions. The works~\cite{QJ-GW-KB-MB:12,QJ-GW-KB-MB:13}
describe a hierarchical setup, with a central coordinator verifying
and assigning reservations, and with vehicles planning their
trajectories locally to platoon and to meet the assigned schedule.
The proposed solution is based on multiagent simulations, an important
difference with respect to our approach. In~\cite{DM-SK:14}, a
polling-systems approach is adopted to assign schedules, and then
optimal trajectories for all vehicles are computed sequentially in
order. Such optimal trajectory computations are costly and depend on
other vehicles' detailed plans, and hence the system is not robust.
In this literature too, a non-collision constraint is imposed on the
vehicles. The work~\cite{HK-DC-PRK:11} is an exception in that it
requires the minimum separation between any two consecutive vehicles
to be a function of the vehicles' velocities. We call such a
constraint as a \emph{safety constraint}, which is a generalization of
a non-collision requirement.

\subsubsection*{Statement of contributions}

Motivated by intelligent management of traffic intersections, we
formulate an optimal control problem for a vehicular string with
safety requirements and finite-time specifications on the approach
time to a target region.  The first contribution is the design of a
distributed control algorithm for the vehicle string that ensures that
the vehicles satisfy the finite-time specifications, while
guaranteeing system-wide safety and subject to speed limits and
acceleration saturation. Additionally, each vehicle seeks to optimally
control its trajectory whenever safety is not immediately at risk.
The algorithm design combines three main elements: an uncoupled
controller ensuring that a vehicle arrives at the intersection at a
designated time when the preceding vehicle is sufficiently far in
front; a safe-following controller ensuring that the vehicle follows
the preceding vehicle safely when the latter is not sufficiently far
in front; and a policy to switch between these two controllers.  The
second contribution is the analysis of the convergence and performance
properties of the vehicle string under the proposed distributed
control design. We provide guarantees on vehicle safe following and
the approach times to the target region. The safe-following
specifications are such that each vehicle maintains at all times
sufficient distance from its preceding vehicle so as to have the
ability to come to a complete stop without collisions irrespective of
the preceding vehicle's control action. This notion has the advantage
of ensuring safety even in cases of communication failures, which may
not be the case for a solution computed with only non-collision
constraints. We also establish that the prescribed approach time of a
vehicle can be met provided it is sufficiently far from the actual
approach time of the previous vehicle in the string. If this is not
the case, then we also provide an upper bound on the difference
between the actual approach times of consecutive vehicles.  The third
contribution is the application to traffic intersection management of
our distributed control design.  We describe how the various
constraints and parameters of the problem can be integrated with an
intelligent intersection manager to provide it with information about
the approach time of the first vehicle in the string and a guaranteed
bound of the occupancy time of the intersection by the string.  Our
simulation study illustrates the results and provides insights on the
algorithm executions and their dependence on the problem parameters.

\subsubsection*{Organization}

The rest of the paper is organized as follows.  Section~\ref{sec:prob}
details the problem formulation and discusses its connection with
intelligent intersection management. Section~\ref{sec:design} presents
the design of the distributed control for the vehicle string and
Section~\ref{sec:evol} derives convergence and performance guarantees
on its executions.  Section~\ref{sec:sim} illustrates our results in a
simulation study. Finally, Section~\ref{sec:conc} contains concluding
remarks and our ideas for future work.

\subsubsection*{Notation}
We present here some basic notation used throughout the paper. We let
$\real$, $\realnegative$, $\integers$, $\integerspositive$, and
$\integersnonnegative$ denote the set of real, negative real, integer,
positive integer, and nonnegative integer numbers, respectively.
Given $a \le b$, $[u]_a^b$ denotes the number $u$ lower and upper
saturated by $a$ and $b$ respectively, i.e.,
\begin{equation*}
  [u]_a^b \triangleq
  \begin{cases}
    a, \quad \text{if } u \leq a,
    \\
    u, \quad \text{if } u \in [a, b],
    \\
    b, \quad \text{if } u \geq b .
  \end{cases}
\end{equation*}

\section{Problem statement}\label{sec:prob}

Consider a string of vehicles on its way to a target region as in
Figure~\ref{fig:vehicle-string}.  The vehicles are labeled
$\{1, \ldots, N\}$ and all have the same length~$L$.  The line segment
represents a road and the target region is the interval $[0, \Delta]$,
representing an intersection. The position of the (front of the)
$j^\text{th}$ vehicle is~$\xv_j$.  We assume that initially the
vehicles are yet to approach the target and hence their initial
positions belong to $\realnegative$. Without loss of generality,
vehicles are indexed in ascending order starting from the vehicle
closest to the target region at the initial time.
\begin{figure}[!htpb]
  \centering
  \includegraphics[scale=0.25]{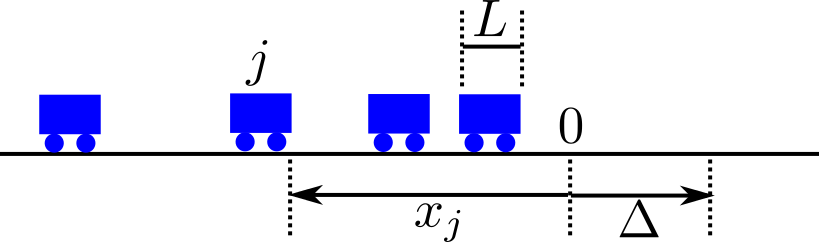}
  \caption{A string of vehicles, each of length $L$, on a road to a
    target at $0$. The position of the (front of the) $j^\text{th}$
    vehicle is $\xv_j$.}\label{fig:vehicle-string}
\end{figure}

\emph{Constraints:} The constraints on the vehicles' motion arise from
their dynamics, specifications regarding their approach to the target
region, and safety requirements. We describe them next.  The dynamics
of vehicle~$j$ is the fully actuated second-order system,
\begin{subequations}\label{eq:vehicle-dyn}
  \begin{align}
    \dxv_j(t) &= \vv_j(t) ,
    \\
    \dvv_j(t) &= \uv_j(t) ,
  \end{align}
\end{subequations}
where $\vv_j \in \real$ is the velocity and $\uv_j(t) \in [u_m, u_M]$,
with $u_m\le 0 \le u_M$, is the input acceleration. The vehicles must
respect a maximum speed limit, $v^M$, imposed on the road ($\vv_j(t)$
must belong to $[0, v^M]$ for all $t \geq 0$).

Each vehicle $j$ is given an exogenous \emph{prescribed approach time}
$\tau_j$ -- the time at which vehicle $j$ is to reach the beginning of
the target region, i.e., the origin. Clearly, an arbitrary set of $\{
\tau_j \}_{j=1}^N$ may not be feasible. We let $\Ta_j$ denote the
actual \emph{approach time} of vehicle $j$ at target region, i.e.,
$\xv_j(\Ta_j) = 0$. Further, we require that the velocity of each
vehicle $j$ at its approach time $\Ta_j$ and subsequently be at least
$\nunom$.  Finally, vehicles are required to maintain a safe distance
between them at all times. Specifically, the minimum separation of
vehicles at any given time must be such that there always exists a
control action for each vehicle to come to a stop safely even without
coordination. Clearly, such a safe-following distance needs to be a
function of the vehicles' velocities, which we denote by
$\Dc(\vv_{j-1}(t), \vv_j(t))$ for the pair of vehicles $j-1$ and
$j$. The formal definition of this function is postponed to the next
section. Then, the safety constraint for $j \in \{2, \ldots, N\}$ is
\begin{equation*}
  x_{j-1}(t) - x_j(t) \geq \Dc(\vv_{j-1}(t), \vv_j(t)), \quad \forall
  t \geq 0.
\end{equation*}

\begin{remark}\longthmtitle{Safety constraints are more robust to loss
    of coordination than non-collision constraints}\label{rem:safety}
  {\rm Typically in the literature, with the exception
    of~\cite{HK-DC-PRK:11}, non-collision constraints are
    imposed. These take the form $x_{j-1}(t) - x_j(t) \geq
    L$. However, non-collision constraints are not robust to loss of
    coordination due to communication failures or otherwise. On the
    other hand, the stricter safety constraints guarantee a higher
    degree of robustness in that there always exists a control action
    for each vehicle to come to a stop safely even with a loss of
    coordination.} \oprocend
\end{remark}

\emph{Objective:} Under the constraints specified above, we seek a
design solution that minimizes the cost function~$C$ modeling 
cumulative fuel cost,
\begin{equation}\label{eq:cost-fun}
  C \triangleq \sum_{j \in \{1, \ldots, N\}} \int_{0}^{\Tx_j} |
  \uv_j | \mathrm{d}t ,
\end{equation}
where $\Tx_j$ is the time at which the vehicle $j$ completely exits
the target region, i.e., $\xv_j(\Tx_j) = \Delta + L$.  This is an
optimal control problem with bounds on the state and control variables
and inter-vehicular safety requirements.  We seek to design a strategy
that allows the vehicle string to solve it in a distributed fashion.
By distributed, we mean that each vehicle $j$ receives information
only from vehicle $j-1$, so that collisions can be avoided and the
algorithm is implementable in real time.  We envision the resulting
distributed vehicular control to be a part of a larger traffic
management solution.

Given the distributed and real-time requirements on the algorithmic
solution, we do not insist on exact optimality and instead focus on
obtaining sub-optimal solutions based on switching between an optimal
control mode and a safe-following mode.  In addition, we also seek to
characterize conditions under which the approach time of each vehicle
is equal to its prescribed approach time (i.e., $\Ta_j = \tau_j$ for
$j \in \until{N}$). Failing feasibility ($\Ta_j \neq \tau_j$ for some
$j$), we aim for our solution to minimize the intersection's
\emph{occupancy time} $\Tx_N - \Ta_1$ of the vehicle string and seek
to provide an upper bound for it.

\begin{remark}\longthmtitle{Connection with intelligent intersection
    management}\label{rem:intersection-management}
  {\rm The motivation for the problem considered here is to enable
    control of a string of computer-controlled, networked vehicles on
    roads with intersections. By networked vehicles, we mean vehicles
    equipped with vehicle-to-vehicle (V2V) and
    vehicle-to-infrastructure (V2I) communication capabilities. We
    envision a system where vehicles communicate their state to an
    \emph{intersection manager} (IM), which then prescribes a schedule
    for the usage of the intersection by the vehicles.  In this paper,
    we do not address the communication and decision making aspects
    related to the interaction between the intersection manager and
    the vehicles.  In the problem posed here, the target region
    corresponds to an intersection, the prescribed approach times are
    given by the IM to the vehicles, and the constraint of a minimum
    approach velocity ensures efficient usage of the intersection. The
    finite-time specifications, bounded controls, speed limits, and
    explicit safety constraints distinguish this work from the
    literature on string stability, which instead focuses on
    asymptotic stability or infinite horizon optimal control with only
    non-collision constraints~\cite{MJ:14-sv, SEL-YZ-KL-JW:15}. We see
    the solution to the problem formulated here as one of the many
    necessary building blocks towards the development of such
    intelligent intersection management capabilities.} \oprocend
\end{remark}

\section{Design of \localvehcontrol controller}\label{sec:design}

In this section we design the distributed vehicle control
termed~\localvehcontrol controller. To do this, we begin by
introducing two useful notions: safe-following distance, as a way of
ensuring safety at present as well as in the future, and relaxed
feasible approach times ignoring safety constraints.

\subsection{Safe-following distance}

The following notion of safe-following distance plays an instrumental
role in guaranteeing the inter-vehicle safety requirement in our
forthcoming developments.

\begin{definition}\longthmtitle{Safe-following
    distance}\label{dfn:sfd}
  The maximum braking maneuver (MBM) of a vehicle is a control action
  that sets its acceleration to $u_m$ until the vehicle comes to a
  stop, at which point its acceleration is set to $0$ thereafter.  
  A quantity $\Dc(\vv_{j-1}(t), \vv_j(t))$ is \emph{a safe-following
    distance} at time~$t$ for the consecutive vehicles $j-1$ and $j$
  if $\xv_{j-1}(t) - \xv_j(t) \geq \Dc(\vv_{j-1}(t), \vv_j(t)) \ge L$
  and, if each of the two vehicles were to perform the MBM, then they
  would be safely separated, $\xv_{j-1} - L \geq \xv_j$, until they
  come to a complete stop. \oprocend
\end{definition}

Given the notion of safe-following distance, we ensure inter-vehicular
safety by requiring for all $j \in \{ 2, \ldots, N \}$ that
\begin{equation}\label{eq:safety-constraints}
  \xv_{j-1}(t) - \xv_j(t) \geq \Dc(\vv_{j-1}(t), \vv_j(t)), \ \forall
  t .
\end{equation}
According to Definition~\ref{dfn:sfd}, a safe-following distance is
not uniquely defined, which in fact provides a certain leeway in
designing the local vehicle control. The following result identifies a
specific safe-following distance.

\begin{lemma}\longthmtitle{Safe-following distance as a function of
    vehicle velocities}\label{lem:sf-dist}
  Let $j-1$ and $j$ be a pair of vehicles, with $j$ following
  $j-1$. Then, the continuous function~$\Dc$ defined by
  \begin{multline}\label{eq:sf-dist}
    \Dc(\vv_{j-1}(t), \vv_j(t)) =
    \\
    L + \max \Big\{ 0, \frac{ 1 }{ -2u_m } \left( (\vv_{j}(t))^2 -
      (\vv_{j-1}(t))^2 \right) \Big\} ,
  \end{multline}
  provides a safe-following distance at time~$t$ for the pair of
  vehicles $j-1$ and~$j$.
\end{lemma}
\begin{proof}
  If a vehicle $j$ with dynamics~\eqref{eq:vehicle-dyn} were to
  perform the MBM at the current time $t$ until it comes to a complete
  stop at $\tstop_{j} = - \vv_j(t)/u_m$, then
  \begin{equation*}
    \xv_j(\tstop_{j}) = \xv_j(t) + \frac{ (\vv_j(t))^2 }{ -2u_m } .
  \end{equation*}
  If $\vv_j(t) \geq \vv_{j-1}(t) \geq 0$, then the safe-following
  distance is found by setting
  \begin{equation*}
    \xv_{j-1}(\tstop_{j-1}) - \xv_j(\tstop_{j}) \geq L .
  \end{equation*}
  If on the other hand $\vv_{j-1}(t) \geq \vv_{j}(t) \geq 0$, then the
  vehicles are in fact closest at time $t$ and the condition
  $\xv_{j-1}(t) - \xv_j(t) \geq L$ is sufficient to ensure subsequent
  safety. Hence~\eqref{eq:sf-dist} provides a safe following distance.
\end{proof}

The safe-following distance function~$\Dc$ defined
in~\eqref{eq:sf-dist} has the following useful monotonicity
properties: if the first argument is fixed, then the function is
monotonically non-decreasing; instead, if the second argument is
fixed, then the function is monotonically non-increasing.

\begin{remark}\longthmtitle{Intra-branch safety under communication
    failure}
  {\rm Note that the safety constraint~\eqref{eq:safety-constraints}
    for each pair of consecutive vehicles on the same branch is more
    than just non-collision constraint. The safety constraints always
    ensure, for each vehicle, the existence of a control action that
    can safely bring the vehicle to a complete stop irrespective of
    the actions of the vehicle preceding it. Thus, in particular, if
    communication were to break down between any pair of consecutive
    vehicles and if the communication failure were detected then the
    following vehicle can safely come to a stop.
 } \oprocend
\end{remark}

\subsection{Feasible approach times ignoring safety
  constraints}\label{sec:relax-feas-app-times}

Here we provide bounds on the feasible approach times by ignoring the
safety constraints.  We rely on these bounds in our controller design
later to ensure the existence of an optimal controller that guarantees
a vehicle arrives at the intersection at a designated time when the
preceding vehicle is sufficiently far in front.

We start by defining the earliest and latest times of approach of each
vehicle at the target region, ignoring other vehicles. Formally, let
$\taue_j$ be the earliest time vehicle $j$ can reach the target region
while ignoring the safety
constraints~\eqref{eq:safety-constraints}. This time can be computed
by considering the trajectory with the initial condition $(x_j(0),
v_j(0))$ and the control policy with maximum acceleration $(u_j =
u_M)$ until $v_j(t) = v_M$ and zero acceleration thereafter. It can be
easily verified that $\taue_j = \Tc( -\xv_j(0), \vv_j(0) )$, where
\begin{align}\label{eq:Tc}
  &\Tc(d, v) \triangleq \notag\\
  &
  \begin{cases}
    \frac{ \sqrt{ 2 u_M d + v^2 } - v }{ u_M }, \ & 2 u_M d \leq
    (v^M)^2 - v^2 , \\
    \frac{ v^M - v }{ u_M } + \frac{ 2 u_M d - (v^M)^2 + v^2 }{ 2 u_M
      v^M }, \ & 2u_M d \geq (v^M)^2 - v^2 .
  \end{cases}
\end{align}
Similarly, let $\taul_j$ be the latest time vehicle $j$ can reach the
target region ignoring the safety constraints by considering
trajectories with maximum deceleration. Note that this could possibly
result in $\taul_i = \infty$.

An important observation is that there might not exist any approach
time $\Ta_j$ during the interval $[\taue_j, \taul_j]$ such that the
minimum approach velocity constraint $\vv_j(\Ta_j) \geq \nunom$ is
satisfied (even when safety constraints are ignored).  The following
result presents a sufficient condition that guarantees the existence
of such approach times even when safety constraints are considered.

\begin{lemma}\longthmtitle{Existence of a feasible approach
    time}\label{lem:initcon-assump}
  Suppose the initial position of vehicle $j \in \until{N}$ satisfies
  \begin{align}\label{eq:suff-cond}
    \xv_j(0) \leq \frac{ (v^M)^2 }{ 2u_m } - \frac{ (\nunom)^2 }{ 2u_M
    } ,
  \end{align}
  then $\taul_j = \infty$, and for any $\tau \in [\taue_j, \infty)$
  there exists a control action that, ignoring the safety
  constraints~\eqref{eq:safety-constraints}, ensures that $\xv_j(\tau)
  = 0$ and $\vv_j(\tau) \geq \nunom$.  Furthermore, if the safety
  constraints~\eqref{eq:safety-constraints} are satisfied initially at
  time $0$, then the set of feasible approach times $[\bar{\tau}_j^e,
  \infty)$ is non-empty and $\bar{\tau}_j^e \geq \taue_j$ for all $j
  \in \until{N}$.
\end{lemma}
\begin{proof}
  The condition on $\xv_j(0)$ implies that vehicle $j$ can come to a
  complete stop, wait for an arbitrarily long time and then accelerate
  to a speed of at least $\nunom$ before arriving at the beginning of
  the target region so that $\taul_j = \infty$. Further, it also means
  that $\vv(\taue_j) \geq \nunom$ under the control action used for
  computing $\taue_j$. Thus, for any $\tau \in [\taue_j, \infty)$
  there exists a control action that, ignoring the safety
  constraints~\eqref{eq:safety-constraints}, ensures that
  $\xv_j(\tau) = 0$ and $\vv_j(\tau) \geq \nunom$.

  If in addition, safety constraints~\eqref{eq:safety-constraints} are
  satisfied initially, then existence of a feasible approach time is
  guaranteed because vehicle $j$ can safely decelerate at the maximum
  rate until it comes to a complete stop while ensuring safety with
  vehicles $j-1$ and $j+1$, then wait for enough time to avoid
  collision with vehicle $j-1$ and accelerate back to $\nunom$ before
  reaching the target region at $\Ta_j$. Finally, if $\Ta_j = t_1$ is
  feasible, then so is $\Ta_j = t_2$ for all $t_2 \geq t_1$ by
  increasing the deceleration time or the wait time.
\end{proof}

Note that $[\bar{\tau}_j^e, \infty)$, the actual set of feasible
approach times for vehicle $j$, depends on all the constraints, the
initial conditions and $\tau_i$ for all the vehicles $i \in \{i,
\ldots, N\}$. As a result, it is not readily computable. In contrast,
the relaxed bound $[\taue_j, \infty)$ is easy to compute and depends
only on the data related to vehicle $j$, which makes it useful in the
control design procedure.

In the rest of the paper we assume that conditions of
Lemma~\ref{lem:initcon-assump} are satisfied for all vehicles $j \in
\until{N}$ (so that feasible approach times satisfying the minimum
velocity requirement are guaranteed to exist) and that $\tau_j \in
[\taue_j, \infty)$, for all $j \in \until{N}$.

\subsection{Controller design}\label{sec:dist-control}

Here we introduce our distributed controller design, which is composed
by three main elements: (i) an uncoupled controller ensuring that the
vehicle arrives at the intersection at a designated time if the
presence of all other vehicles is ignored. This controller is applied
when the preceding vehicle is sufficiently far in front, (ii) a
safe-following controller ensuring that the vehicle follows the
preceding vehicle safely when the latter is not sufficiently far in
front; and (iii) a rule to switch between the two controllers. 

\subsubsection{Uncoupled controller}

We let
\begin{equation*}
  (t, \xv_j, \vv_j) \mapsto \guc(\tau_j, t, \xv_j,
  \vv_j) 
\end{equation*}
be a feedback controller that ensures $\xv_j(\tau_j) = 0$ for the
dynamics~\eqref{eq:vehicle-dyn} starting from the current state
$(\xv_j(t), \vv_j(t))$ at time $t$, respecting the control and
velocity constraints, but not necessarily the inter-vehicle safety
constraints. We refer to it as the \emph{uncoupled} controller. Here,
we let $\guc$ be the optimal feedback controller that is obtained in
the sense of receding horizon control (RHC)~\cite{JBR:00}. The cost
function that is minimized in an open-loop manner as part of RHC is
\begin{align*}
  \int_{t}^{\tau_j} |\uv_j(s)| ds .
\end{align*}
For simplicity we restrict the candidate velocity profiles, each of
which determines uniquely a control trajectory $\uv_j$, to the two
classes shown in Figure~\ref{fig:pwc-control}. Thus, the optimization
variables in RHC are $a_1$, $a_2$ (the areas of the indicated
triangles), $\vv_j(\tau_j)$, $\nu^l$ and $\nu^u$. The constraints are
$\nu^l \in [0, \vv_j(t)]$, $\nu^u \in [\vv_j(t), v^M]$,
$\vv_j(\tau_j) \in [\nunom, v^M]$, $a_1, a_2 \geq 0$ and that the
total area under the curve (corresponding to the distance traveled) is
equal to $-\xv_j(t)$.

\begin{remark}\longthmtitle{Alternative implementation of the
    uncoupled feedback
    controller}\label{rem:alternative-implementation}
  {\rm In contrast to the receding horizon approach, the uncoupled
    controller $\guc$ may also be implemented as a feedback
    controller. This may be done by analytically solving off-line the
    initial optimal control action as a function of
    $\tau_j, t, x_j, v_j$. In our particular problem, due to the
    specific structure of the candidate solutions, shown in
    Figure~\ref{fig:pwc-control}, this analytical computation could be
    done exhaustively. Such computations would result in a $\guc$ that
    is a switched controller (with analytical expressions) as a
    function of $\tau_j, t, x_j, v_j$. Further, as a vehicle's state
    $(x_j, v_j)$ and time $t$ evolve continuously with time $t$, there
    are only a few modes to which $\guc$ could switch to at any given
    time. Thus, in practice, such a switched feedback controller
    $\guc$ could be implemented quite efficiently.
  } \oprocend
\end{remark}

\begin{figure}[!htb]
  \centering \subfigure[\label{fig:pwc_down}]
  {\includegraphics[width=0.35\textwidth]{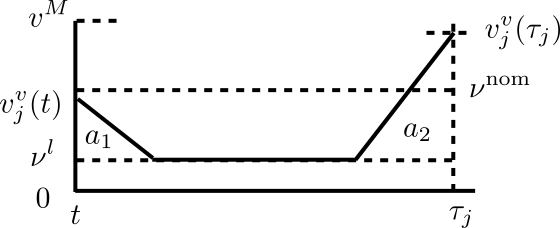}}

  \subfigure[\label{fig:pwc_up}]
  {\includegraphics[width=0.35\textwidth]{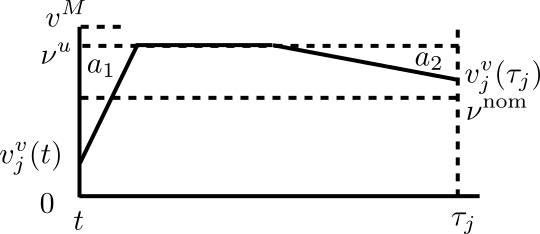}}
  \vspace*{-1ex}
  \caption{(a) and (b) show the two classes of open-loop candidate
    velocity profiles. The optimal profile takes one of the two forms
    depending on the velocity $\vv_j(t)$, $\nunom$, $v^M$, $\tau_j$
    and the distance to go~$-\xv_j(t)$. The uncoupled feedback
    controller $\guc$ may be obtained by executing the open-loop
    optimal control in a receding horizon manner.
  }\label{fig:pwc-control}
\end{figure}

\begin{remark}\longthmtitle{Optimality of the
    controller}\label{rem:optim-pwc}
  {\rm Assuming there exists a feasible controller that ensures the
    vehicle $j$ approaches the intersection at $\tau_j$ with a minimum
    velocity of $\nunom$, ignoring the safety constraints and given
    the current time $t$ and the vehicle state $(\xv_j(t), \vv_j(t))$,
    then there exists an optimal solution with piecewise-constant-rate
    velocity profiles as shown in Figure~\ref{fig:pwc-control}. We can
    see this statement to be true by observing that in a given time
    $\tau_j - t$, the minimum and maximum travel distances are
    obtained with velocity profiles belonging to the family depicted
    in Figure~\ref{fig:pwc-control}, and that every other intermediate
    travel distance is obtained by a continuous variation of the
    velocity profiles within the family. } \oprocend
\end{remark}

It is worth noticing that the control $\guc$ assumes the presence of
no other vehicles. Thus, the actual approach time, $\Ta_j$, of the
vehicle $j$ may be later than $\tau_j$. Note that $\Ta_j$ could not be
earlier than $\tau_j$ because each vehicle $j$ receives information
only from vehicle $j-1$ and it never attempts to approach the
intersection before~$\tau_j$.

Note that a feasible $\guc$ exists for each vehicle at $t = 0$. This
is because under the conditions of Lemma~\ref{lem:initcon-assump}, the
feasible approach times ignoring the safety constraints is $[
\taue_{j}, \infty )$ and we assume that $\tau_j \in [ \taue_{j},
\infty )$. However, at a future time $t$, such a feasible $\guc$ might
not exist because the vehicle is slowed down by preceding vehicles and
no control exists to ensure $\Ta_j = \tau_j$ along with the other
constraints. Additionally, for $t > \Ta_j$, i.e., after the vehicle
enters the target region, the optimal controller is not well defined
and does not exist.  As a shorthand notation, we use $\exists \Fc_j$
(respectively $\nexists \Fc_j$) to denote the existence (respectively,
lack thereof) of an optimal uncoupled control. In order for the
control $\guc$ to be well defined at all times and for all states, we
extend it as
\begin{equation*}
  \guc(\tau_j, t, \xv_j, \vv_j) \triangleq u_M, \quad \text{if }
  \nexists \Fc_j .
\end{equation*}

\subsubsection{Controller for safe following}

As mentioned earlier, this controller is applied only when a vehicle
is sufficiently close to the vehicle preceding it. Besides maintaining
a safe-following distance, the controller must also ensure that the
resulting evolution of the vehicles does not result in undue delays in
approach times. Here, we present a design to achieve these goals. For
a pair of vehicles $j-1$ and $j$, we define the \emph{safety ratio} as
\begin{equation}\label{eq:sigma}
  \sigma_j(t) \triangleq \frac{ \xv_{j-1}(t) - \xv_j(t) }{ 
    \Dc(\vv_{j-1}(t), \vv_j(t)) } ,
\end{equation}
which is the ratio of the actual inter-vehicle distance to the
safe-following distance. Hence, the
requirement~\eqref{eq:safety-constraints} can be equivalently
expressed as stating that $\sigma_j (t)$ should remain above $1$ at
all times. Notice from the definition~\eqref{eq:sf-dist} of the
safe-following distance that, if $\vv_{j-1}(t) > \vv_j(t)$, then
$\sigma_j(t)$ increases at time $t$ and safety is guaranteed. Thus, it
is sufficient to design a controller that ensures safe following when
$\vv_j(t) \geq \vv_{j-1}(t)$. For vehicle $j$, we denote $\zeta_j
\triangleq ( \vv_{j-1}, \vv_j, \sigma_j )$. Define the
\emph{unsaturated} controller~$\gus$ by
\begin{align*}
  \gus(\zeta_j, &\uv_{j-1}) \triangleq
  \\
  &
  \begin{cases}
    \uv_{j-1}, & \text{if } \vv_j = 0 ,
    \\
    \left( \frac{ \vv_{j-1} }{ \vv_j } \left( 1 + \sigma_j \frac{
          \uv_{j-1} }{ -u_m } \right) -1 \right) \left( \frac{ -u_m }{
        \sigma_j } \right), & \text{if } \vv_j > 0 .
  \end{cases}  
\end{align*}
The rationale behind this definition is as follows. Since it is
sufficient to design a controller that ensures safe following when
$\vv_j(t) \geq \vv_{j-1}(t)$, if $\vv_j = 0$, then we need to consider
only the case $\vv_{j-1} = 0$. In this case, the definition of $\gus$
ensures that the vehicle $j$ stays at rest as long as vehicle $j-1$ is
at rest, and starts moving only when $j-1$ starts moving
again. Further, since the relative velocity and acceleration in this
case would be zero, we see that $\sigma_j$ stays constant. On the
other hand, if the vehicle is moving, $\vv_j > 0$, then $\guc$ is
designed to make sure that $\sigma_j$ stays constant (we show this
formally in Lemma~\ref{lem:coupling-set}), thus ensuring safety.
However, in this latter case, $\gus$ might cause $\vv_j$ to
exceed~$v^M$. Further, we would like the vehicle to continue using the
optimal uncoupled controller if it does not affect the safety by
decreasing $\sigma_j$.  These considerations motivate our definition
of the \emph{safe-following} controller as
\begin{multline}\label{eq:sf-control}
  \gsf(t, \zeta_j, \uv_{j-1}) \triangleq
  \\
  \min \lbrace \guc(\tau_j, t, \xv_j, \vv_j), \gus(\zeta_j, \uv_{j-1})
  \rbrace .
\end{multline}

\subsubsection{\localvehcontrol controller}

Here, we design the local vehicle controller by specifying a rule to
switch between the uncoupled controller $\guc$ and the safe-following
controller~$\gsf$. To make precise whether two vehicles are
sufficiently far from each other, we introduce the \emph{coupling set}
$\Cc_s$ defined by
\begin{equation}\label{eq:coupling-set}
  \Cc_s \triangleq \{ ( \mathrm{v}_1, \mathrm{v}_2, \sigma ) :
  \mathrm{v}_2 \geq \mathrm{v}_1 \text{ and } \sigma \in [ 1, \sigma_0
  ] \} ,
\end{equation}
where $\sigma_0 > 1$ is a design parameter.  The value of this
parameter marks when the safety ratio is considered to be sufficiently
close to $1$ that action is required to prevent inter-vehicle
collision.  This criteria is more conservative (resp. aggressive) the
further (resp. closer) $\sigma_0$ is to $1$. If $\zeta_j \in \Cc_s$,
then vehicle $j$ is going at least as fast as the vehicle in front of
it, and their safety ratio is sufficiently close to $1$ that action is
required.  With this in mind, we define the \localvehcontrol
controller for vehicle $j$, to make sure it uses the safe-following
controller when it is in the coupling set, and the uncoupled
controller otherwise. Formally,
\begin{align}\label{eq:vehicle-control}
  \uv_j(t) =
  \begin{cases}
    \guc, &\text{if } \zeta_j \notin \Cc_s, \ \vv_j < v^M ,
    \\
    [ \guc ]_{u_m}^0, &\text{if } \zeta_j \notin \Cc_s, \ \vv_j = v^M ,
    \\
    \gsf, &\text{if } \zeta_j \in \Cc_s, \ \vv_j < v^M ,
    \\
    [ \gsf ]_{u_m}^0, &\text{if } \zeta_j \in \Cc_s, \ \vv_j = v^M .
  \end{cases}
\end{align}
Note that $[ \guc ]_{u_m}^0 \neq \guc$ only if $\nexists \Fc_j$.

\section{Evolution of the vehicle string}\label{sec:evol}

In this section, we analyze the evolution of the vehicle string under
the distributed controller designed in
Section~\ref{sec:dist-control}. Specifically, we characterize to what
extent the controller allows the vehicles to meet the specifications
on the vehicle string regarding safety and approach to the target
region.

\subsection{Vehicle behavior under safe following}

Here, we study the dynamical behavior of the vehicles when they are in
the coupling set, i.e., when they operate under the safe-following
controller~$\gsf$ in~\eqref{eq:sf-control}.  The next result
identifies conditions under which the safety ratio remains constant
and the unsaturated controller exceeds the maximum acceleration.

\begin{lemma}\longthmtitle{Vehicle behavior in the coupling
    set}\label{lem:coupling-set}  
  For $j \in \ntil{N}$, let $t \in \realnonnegative$ such that
  $\zeta_j(t) = (v_{j-1}(t),v_{j}(t),\sigma_j(t)) \in \Cc_s$ and
  $\uv_{j-1}(t) \in [u_m, u_M]$. Then, the following hold:
  \begin{enumerate}
  \item $\gus(\zeta_j, \uv_{j-1}) \in [u_m, u_M]$,
  \item If $\vv_j < v^M$ and $\gsf(t, \zeta_j, \uv_{j-1}) =
    \gus(\zeta_j, \uv_{j-1})$ or if $\vv_j = v^M$ and $\gsf(t,
    \zeta_j, \uv_{j-1}) = [\gsf(t, \zeta_j, \uv_{j-1})]_{u_m}^0 =
    \gus(\zeta_j, \uv_{j-1})$, then $\dot{\sigma}_j = 0$,
  \item If $\vv_j = \vv_{j-1} \geq 0$ and $\gsf(t, \zeta_j, \uv_{j-1})
    = \gus(\zeta_j, \uv_{j-1})$, then $\dot{\sigma}_j = 0$ and $\uv_j
    = \uv_{j-1}$,
  \item If $\vv_j = v^M$, then $\gus(\zeta_j, \uv_{j-1}) \geq
    [\gus(\zeta_j, \uv_{j-1})]_{u_m}^0 = 0$ only if
    \begin{equation*}
      \vv_{j-1} \geq \ulinevv \triangleq
      \frac{ -u_m v^M }{ -u_m + \sigma_0 u_M } .
    \end{equation*}
  \end{enumerate}
\end{lemma}
\begin{proof}
  For the sake of conciseness, we drop the arguments of the functions
  wherever it causes no confusion.

  (a) For $\vv_j = 0$, the claim readily follows from the definition
  of $\gus$. For fixed $\sigma_j \geq 1$, $\vv_j \geq \vv_{j-1} \geq
  0$ and $\vv_j > 0$, we see that $\gus$ is maximized and minimized
  when $\uv_{j-1} = u_M$ and $\uv_{j-1} = u_m$, respectively. The
  result then follows by observing, after some computations, that
  $\gus(\zeta_j, u_M) - u_M \leq 0$ and $\gus(\zeta_j, u_m) - u_m \geq
  0$.

  (b) and (c) From~\eqref{eq:sigma} observe that
  \begin{align*}
    \dot{\sigma}_j &= \frac{ \vv_{j-1} - \vv_j - \sigma_j
      \dot{\Dc}(\vv_{j-1}(t), \vv_j(t)) }{ \Dc(\vv_{j-1}(t), \vv_j(t))
    }
    \\
    &= \frac{ \vv_{j-1} - \vv_j - \frac{ \sigma_j }{ -u_m } ( \vv_j
      \uv_j - \vv_{j-1} \uv_{j-1} ) }{ \Dc(\vv_{j-1}(t), \vv_j(t)) }
  \end{align*}
  where we have used the fact that $\vv_j \geq \vv_{j-1}$ in the
  coupling set~$\Cc_s$. Claim~(b) now follows by substituting $\uv_j =
  \gsf = \gus$ and using the definition of $\gus$. A similar argument
  can be used to show claim~(c).

  (d) Setting $\vv_j = v^M$ in the definition of $\gus$ and using the
  fact that $\gus \geq 0$, we have
  \begin{align*}
    \vv_{j-1} \geq \frac{ -u_m v^M }{ -u_m + \sigma_j \uv_{j-1} } .
  \end{align*}
  To obtain the necessary condition on $\vv_{j-1}$, we set $\uv_{j-1}
  = u_M$ and $\sigma_j = \sigma_0$, the maximum values for each.
\end{proof}

Lemma~\ref{lem:coupling-set} identifies conditions under which we can
describe the behavior of the vehicles when they are in the coupling
set.  The unsaturated controller has been designed so as to ensure
that the safety ratio is maintained at a constant level. We know from
claim (a) that $\gus$ respects the control constraints. Thus, in claim
(b), we see that when $\vv_j < v^M$ and $\gsf = \guc$, the control
action would not violate the velocity constraints and hence ensures
that the safety ratio remains constant, i.e., $\dot{\sigma}_j = 0$.
Similarly, when $\vv_j = v^M$, if the control action $\gsf = \guc$ and
it happens to be non-positive, then again the velocity constraints
would not be violated and we have $\dot{\sigma}_j = 0$. Claim (c) is a
special case of (b) with $\vv_j = \vv_{j-1}$. In this special case, we
further have that the relative acceleration, and hence also the
relative velocity, are zero. Finally, claim (d) is a necessary
condition on the velocity of vehicle $j-1$ for $\gus$ and the
saturated $[\gus]_{u_m}^0$ to differ when $\vv_j = v^M$. In other
words, if $\vv_{j-1} \leq \ulinevv$, then from claims (d) and (b) we
have that $\dot{\sigma}_j = 0$. It is worth noting that, while
considered separately the conditions in each case of
Lemma~\ref{lem:coupling-set} might appear restrictive, when considered
all together they paint a fairly general picture.

As we have seen in Lemma~\ref{lem:coupling-set} and its
interpretation, the unsaturated controller $\gus$ has been designed
with the aim of ensuring that the safety ratio remains constant. Thus,
it would be interesting to determine conditions under which the use of
$\gus$ guarantees string stability. The next result states that, in
fact, this is the case in the absence of velocity constraints and
assuming that the leading vehicle's velocity is uniformly upper
bounded in time.

\begin{proposition}\longthmtitle{Unsaturated controller $\gus$ ensures
    string stability in the absence of velocity
    constraints} \label{prop:string-stab}
  Consider two vehicles $j-1$ and $j$, with vehicle $j$ following
  $j-1$. Suppose that the velocity of the leading vehicle $j-1$ is
  uniformly lower and upper bounded in time by $0$ and a
  constant $\bar{\Vc}$, respectively. Further suppose that no bound on
  the velocity of vehicle $j$ is imposed. If the initial condition are
  such that $\vv_j(0) \geq \vv_{j-1}(0) \geq 0$ and
  $\sigma_j(0) \in [ 1, \sigma_0 ]$, then the control policy
  $u_j = \gus$ ensures that:
  \begin{enumerate}
  \item safety is guaranteed for all time $t$ by ensuring
    $\sigma_j(t) = \sigma_j(0) \geq 1$,
  \item $\displaystyle \vv_j \leq \sqrt{ (v_j(0))^2 -
        (v_{j-1}(0))^2 + \bar{\Vc}^2 }$, for all $t \geq 0$, 
  \item $\vv_j$ asymptotically approaches $\vv_{j-1}$,
  \item $\xv_{j-1} - \xv_{j}$ asymptotically converges to
    $\sigma_j(0) L \leq \sigma_0 L$.
  \end{enumerate}
\end{proposition}
\begin{proof}
  If there exists $s$ such that $\vv_j(s) = \vv_{j-1}(s) = 0$, then
  the result is trivially true because, by definition, $\uv_j(t) =
  \uv_{j-1}(t)$ for all $t \geq s$. Thus, in the remainder of the
  proof we assume that $\vv_j(t) > 0$ for all $t$ in addition to the
  fact that $\vv_j(t) \geq \vv_{j-1}(t) \geq 0$.
  Lemma~\ref{lem:coupling-set}(b) directly guarantees claim~(a).

  Now, let $e_j \triangleq \vv_j - \vv_{j-1}$ and observe that
    \begin{align}\label{eq:ejdot}
      \dot{e}_j &= \uv_j - \uv_{j-1} = \gus - \uv_{j-1} \notag
      \\
      &= - \left( \frac{ -u_m }{ \sigma_j } \right) \left( \frac{ e_j
        }{ \vv_j } \right) \left( \frac{ \sigma_j \uv_{j-1} - u_m }{
          -u_m } \right) \notag
      \\
      & = - \left( \frac{ \sigma_j(0) \uv_{j-1} - u_m }{ \sigma_j(0)
          \vv_j } \right) e_j ,
    \end{align}
    where in the last step we used $\sigma_j = \sigma_j(0) \geq 1$.
    Based on the discussion above, we exclude the case of $v_j = 0$.
    Thus, under such conditions, $e_j = 0$ is invariant. This implies
    that $e_j \geq 0$ for all $t \geq 0$ given the assumption on the
    initial condition. Also note that $u_m < 0$, $\uv_{j-1} \geq u_m$
    and $\vv_j \geq 0$.  As a result the following observations
      hold:
    \begin{itemize}
    \item $\dot{e}_j > 0$ only if
      $\displaystyle u_{j-1} \in [u_m, { u_m }/{ \sigma_j(0) } )$,
    \item $\dot{v}_j = u_j > 0$ only if $\displaystyle u_{j-1} > 0 > {
        u_m }/{ \sigma_j(0) }$,
    \item $\dot{v}_j = u_j > 0$ only if $\dot{e}_j < 0$.
    \end{itemize}
    The first observation follows from~\eqref{eq:ejdot}. The second is
    obtained by setting $\gus > 0$, which implies
    \begin{equation*}
      u_{j-1} > \frac{ -u_m e_j }{ \sigma_j(0) v_{j-1} } \geq 0 .
    \end{equation*}
    The final observation follows from the first two. Thus, at any
    given time, we have at least one of $e_j$ or $v_j$
    non-increasing. Motivated by this, we next show $v_j$ is bounded.
    Indeed, from~\eqref{eq:sigma} and the fact $\sigma_j(t) =
    \sigma_j(0)$,
    \begin{equation*}
      \sigma_j(0) \dot{\Dc}(\vv_{j-1}(t), \vv_j(t)) = - e_j \leq 0 ,
    \end{equation*}
    which implies that $\Dc$ is non-increasing because $e_j \geq 0$.
    Then, from~\eqref{eq:sf-dist} and the fact
    $v_j(t) \geq v_{j-1}(t)$ for all $t$,
    \begin{equation*}
      (v_j(t))^2 \leq (v_j(0))^2 - (v_{j-1}(0))^2 + (v_{j-1}(t))^2 ,
    \end{equation*}
    from which claim (b) follows.

    Since $v_{j-1} \leq v_j$ for all $t \geq 0$, the inter-vehicular
    distance $(x_{j-1} - x_j)$ is monotonically
    non-increasing. Further, since $\sigma_j = \sigma_j(0)$, we see
    from~\eqref{eq:sigma} and~\eqref{eq:sf-dist} that
    $(x_{j-1} - x_j)$ is uniformly lower bounded by $\sigma_j(0) L$.
    Thus,
    \begin{equation*}
      x_{j-1}(t) - x_j(t) = - \int_0^t e_j(s) \mathrm{d}s + x_{j-1}(0)
      - x_j(0)
    \end{equation*}
    must asymptotically converge to a finite constant. Now, notice
    that $| \dot{e}_j |$ is uniformly upper bounded due to the bounds
    on $u_{j-1}$ and $u_j$. Hence, $e_j$ is uniformly
    continuous. Then, claim~(c) follows from Barbalat's Lemma,
    cf.~\cite{HKK:02}.

  Finally, we know that $\sigma_j(t) = \sigma_j(0) \leq \sigma_0$ for
  all $t$. Further, as $\vv_j$ approaches $\vv_{j-1}$, the
  safe-following distance $\Dc(\vv_{j-1}(t), \vv_j(t))$ approaches $L$
  (cf.~\eqref{eq:sf-dist}). Then, claim (d) follows from the
  definition of the safety ratio~\eqref{eq:sigma}.
\end{proof}

As a consequence of Proposition~\ref{prop:string-stab}, any string
with a finite number of vehicles can be stabilized using the
controller $\gus$ if each pair of consecutive vehicles is initially in
the coupling set and if the velocity of the first vehicle in the
string is uniformly upper bounded in time.

The next result states that if at any time instant the optimal
controller does not exist (because the vehicle has been slowed down by
preceding vehicles), then a vehicle not in the coupling set moves at
the maximum speed.

\begin{lemma}\longthmtitle{If the uncoupled optimal controller does not
    exist then the vehicle exits the coupling set at maximum
    speed}\label{lem:not-in-coupling-set} 
  Let $t_1$ be any time such that $\zeta_j(t_1) \in \Cc_s$ and
  $\zeta_j(t) \notin \Cc_s$ for $t \in (t_1, t_1 + \delta)$ for some
  $\delta > 0$. If $\nexists \Fc_j$ at time~$t_1$, then $\vv_j(t) =
  v^M$ for all $t \in [t_1, t_1 + \delta)$.
\end{lemma}
\begin{proof}
  Under the hypotheses of the result, and as a consequence of
  Lemma~\ref{lem:coupling-set}(c), the only way $\vv_j(t_1) =
  \vv_{j-1}(t_1)$ is possible is if $\gsf = \guc < u^M$ at $t_1$,
  i.e., $\exists \Fc_j$. However, by assumption $\nexists \Fc_j$ at
  time $t_1$, meaning $\guc = u^M$. Thus, it must be that $\vv_j(t_1)
  > \vv_{j-1}(t_1)$. By definition of $t_1$, we then conclude that
  $\sigma_j(t_1) = \sigma_0$. Next, at $t_1$, since $\nexists \Fc_j$
  it means $\guc = u_M$ and thus $\gsf = \gus$. Then, from
  Lemma~\ref{lem:coupling-set}(b), we see that $\vv_j(t_1) < v^M$ is
  not possible and that in fact $\vv_j(t_1) = v^M$ and $\gsf = \gus >
  [\gsf]_{u_m}^0 = 0$. During the interval $(t_1, t_1 + \delta)$, we
  see from the second case of~\eqref{eq:vehicle-control} that $\uv_j =
  [ \guc ]_{u_m}^0 = [ u_M ]_{u_m}^0 = 0$, which proves the result.
\end{proof}

This result is useful in our forthcoming analysis to bound the arrival
times of consecutive vehicles to the target region.

\subsection{Guarantees on vehicle approach times to target
  region}\label{sec:app=times}

In this section we provide guarantees on the vehicle approach times to
the target region under the \localvehcontrol controller.  Our main
result states that the prescribed approach time of a vehicle can be
met provided it is sufficiently far from the actual approach time of
the previous vehicle in the string. If this is not the case, then the
result provides an upper bound on the difference between the actual
approach times.

To precisely quantify the upper bound, we introduce below the
quantity~$ \Tiat$. To justify its definition, we first need to
introduce some useful concepts. Let
$\Dcnom \triangleq \Dc(\nunom, v^M)$, which has the interpretation of
a safe inter-vehicle distance given a vehicle is traveling at the
maximum allowed speed~$v^M$ and the vehicle preceding it is traveling
at a speed greater than or equal to~$\nunom$. Recall that we require
that each vehicle maintain a velocity of at least $\nunom$ as it
approaches the target region and subsequent to it.  Given the
monotonicity properties of the safe-following distance function $\Dc$
defined in~\eqref{eq:sf-dist}, we see that $\Dcnom$ is an upper bound
on the safe-following distance for any pair of consecutive vehicles
$j-1$ and $j$ for all time subsequent to the approach time of vehicle
$j-1$, i.e., for all $t \geq \Ta_{j-1}$. Thus, if the following
vehicle $j$ is within the coupling set with vehicle $j-1$ at the time
of its approach, $\Ta_j$, then we show in the proof of the next result
that the inter-approach time $\Ta_j - \Ta_{j-1}$ is upper bounded by
$\sigma_0 \Tnom$, where
\begin{align}\label{eq:Tnom}
  \Tnom \triangleq \Dcnom / \nunom ,
\end{align}
which we call the \emph{nominal safe inter-vehicle approach time}.

If, instead, vehicles $j-1$ and $j$ do not belong to the coupling set
at $\Ta_j$, and $\Ta_j > \tau_j$, then from
Lemma~\ref{lem:not-in-coupling-set} we know that if $t_e$ is the
moment when $\nexists \Fc_j$ and vehicle $j$ exits (never to enter
again) the coupling set with vehicle $j-1$, then $\vv_j = v^M$ for all
$t \in [t_e, \Ta_j]$. Note also that, by definition,
$\sigma_j(t_e) = \sigma_0$. Thus, letting $\xv_{j-1}(t_e) = - d$ and
$\vv_{j-1}(t_e) = v$, we see from~\eqref{eq:sigma} that
$\xv_j(t_e) = - (d + \sigma_0 \Dc(v, v^M))$. Hence,
\begin{equation*}
  \Ta_j = \frac{ d + \sigma_0 \Dc(v, v^M) }{ v^M } ,
\end{equation*}
and as a result $\Ta_j - \Ta_{j-1} \leq \Lc(d,v)$, with
\begin{equation}\label{eq:Lc}
  \Lc(d,v) \triangleq \frac{ d + \sigma_0 \Dc(v, v^M) }{ v^M } -
  \Tc(d,v) ,
\end{equation}
where $\Tc$ is as defined in~\eqref{eq:Tc} and gives the earliest
possible approach time given the distance to go and the current
velocity. With this discussion in place, we are ready to define
\begin{equation}\label{eq:Tiat}
  \Tiat \triangleq \max \{ \sigma_0 \Tnom, \max_{ \substack{d \geq
      \frac{ (\nunom)^2 - v^2 }{ 2u_M }, \\ v \in [\ulinevv, \nunom]} }
  \Lc(d,v) \} ,
\end{equation}
where $\ulinevv$ is given in Lemma~\ref{lem:coupling-set}(d).  This
time also plays a key role in uniformly upper bounding (independently
of the initial conditions) the difference between the actual approach
times of consecutive vehicles if they are not in the coupling set when
reaching the target region. The constraints on $d$ and $v$
in~\eqref{eq:Tiat} essentially constitute, as shown in the proof of
the next result, a sufficient condition for the occurrence of the case
in which the vehicles $j-1$ and $j$ do not belong to the coupling set
at $\Ta_j$, and $\Ta_j > \tau_j$.

We are now ready to state formally the first result of this section.

\begin{proposition}\longthmtitle{Inter-approach times of vehicles at
    the target region}\label{prop:int-app-times}
  For any vehicle $j \in \{ 2, \ldots, N \}$, suppose
  that~\eqref{eq:suff-cond} holds, $\tau_j \in [\taue_j, \taul_j]$,
  and $\vv_{j-1}(\Ta_{j-1}) \geq \nunom$. Then, $\vv_j(\Ta_j) \geq
  \nunom$ and
  \begin{enumerate}
  \item if $\tau_j - \Ta_{j-1} \leq \Tiat$, then $ \Ta_j - \Ta_{j-1}
    \leq \Tiat$,
  \item if $\tau_j - \Ta_{j-1} \geq \Tiat$, then $\Ta_j = \tau_j$.
  \end{enumerate}
\end{proposition}
\begin{proof}
  First note that initially at $t = 0$, Lemma~\ref{lem:initcon-assump}
  guarantees that $\exists \Fc_j$. Next, notice from the definition of
  the controller~\eqref{eq:vehicle-control} that $\uv_{j}(t) \leq
  \guc$ for all $t \geq 0$. Further notice that if at some time $t_1$,
  $\nexists \Fc_{j}$, then it remains $\nexists \Fc_{j}$ for all $t
  \geq t_1$ for otherwise it means there exists some control policy
  starting from $t = t_1$ such that $\Ta_{j} = \tau_{j}$ and
  $\vv_{j}(\Ta_{j}) \geq \nunom$ and Remark~\ref{rem:optim-pwc}
  guarantees $\exists \Fc_{j}$ at $t = t_1$. From this discussion, we
  deduce that $\Ta_{j} \geq \tau_{j}$ for each vehicle $j$.

  (a) There are two cases - either the uncoupled optimal controller
  exists until the vehicle reaches the target region or it becomes
  infeasible earlier. We consider each of these cases separately. In
  the first case, notice that for any vehicle $j \in \{ 2, \ldots, N
  \}$, if $\exists \Fc_{j}$ at $t = \Ta_{j}$, then it follows from the
  definition of $\Ta_{j}$ that $\Ta_{j} = \tau_{j}$ and $\vv_{j}(
  \Ta_{j} ) \geq \nunom$, which means claim (a) is true in the first
  case.

  Next, we consider the case when $\nexists \Fc_{j}$ first occurs at
  some time $t_f < \Ta_{j}$. Clearly, $\zeta_j (t_f) \in \Cc_s$. Now,
  there are two sub-cases - either $\zeta_j(\Ta_j) \in \Cc_s$ or
  $\zeta_j(\Ta_j) \notin \Cc_s$. In the first sub-case, we have by
  definition that $\sigma_j(\Ta_j) \leq \sigma_0$ and
  $\vv_j(\Ta_j) \geq \vv_{j-1}(\Ta_j)$. Then, the fact that
  $\vv_{j-1}(t) \geq \nunom$ for all $t \geq \Ta_{j-1}$ implies
  \begin{align*}
    \xv_{j-1}(\Ta_j) - \xv_j(\Ta_j) &= \sigma_j(\Ta_j) \cdot \Dc(
    \vv_{j-1}(\Ta_j), \vv_j(\Ta_j) )
    \\
    &\leq \sigma_0 \cdot \Dcnom ,
  \end{align*}
  where we have used the definition of $\Dcnom$ and the monotonicity
  properties of the safe-following distance function $\Dc$ in deriving
  the inequality. Now, imagine a virtual particle rigidly fixed to
  vehicle $j-1$ at a distance of $\sigma_0 \Dcnom$ behind it. Since
  $\vv_{j-1}(t) \geq \nunom$ for all $t \geq \Ta_{j-1}$, we can then
  conclude that
  $\Ta_{j} - \Ta_{j-1} \leq \sigma_0 \frac{ \Dcnom }{ \nunom } =
  \sigma_0 \Tnom \le \Tiat$.

  We are then left with the sub-case when
  $\zeta_j(\Ta_j) \notin \Cc_s$.  Thus, suppose that there exists
  $t_e \geq t_f$ such that $\zeta_j(t) \notin \Cc_s$ for all
  $t \in (t_e, \Ta_j]$ and $\zeta_j(t_e) \in \Cc_s$. From
  Lemma~\ref{lem:not-in-coupling-set}, it follows that
  $\vv_j(t) = v^M$ for all $t \in [t_e, \Ta_j]$. Thus, as we have seen
  in~\eqref{eq:Lc}, $\Ta_j - \Ta_{j-1} \leq \Lc(d,v)$ with
  $\xv_{j-1}(t_e) = - d$ and $\vv_{j-1}(t_e) = v$. Thus, now it
  remains to justify the constraints on $d$ and $v$
  in~\eqref{eq:Tiat}. Given the assumption that
  $\vv_{j-1}(\Ta_{j-1}) \geq \nunom$ it follows that
  $d \geq \frac{ (\nunom)^2 - v^2 }{ 2u_M }$, which is the minimum
  distance traversed as the velocity of a vehicle evolves from $v$ to
  $\nunom$. Next, by the definition of $t_e$, note that
  $\sigma_j(t_e) = \sigma_0$ and $\sigma_j(t) > \sigma_0$ for all
  $t \in (t_e, \Ta_j]$, implying that $\dot \sigma_j(t_e)>0$. From
  Lemma~\ref{lem:coupling-set}(b)-(d), we then deduce that
  $\vv_{j-1}(t_e) \geq \ulinevv$. Finally, notice that
  \begin{align*}
    \xv_{j-1}(\Ta_{j}) - \xv_j(\Ta_{j}) & \leq \xv_{j-1}(t_e) -
    \xv_j(t_e)
    \\
    & = \sigma_0 \Dc( \vv_{j-1}(t_e), \vv_j(t_e) ) ,
  \end{align*}
  where the inequality follows from $\vv_{j-1}(t) \leq \vv_j(t) = v^M$
  for all $t \in [t_e, \Ta_j]$. Consequently, if $\vv_{j-1}(t_e) = v
  \geq \nunom$, then $\Dc( v, v^M ) \leq \Dc( \nunom, v^M )$ and hence
  we deduce $\Ta_{j} - \Ta_{j-1} \leq \sigma_0 \Tnom$, which justifies
  the final constraint in~\eqref{eq:Tiat} and hence proves claim (a).
 
  (b) The main argument for the proof of this claim is that the
  uncoupled optimal controller exists until the vehicle reaches the
  target region, which we show by contradiction. Suppose that
  $\nexists F_j$ at $\Ta_j$. Then as in the proof of claim (a), we see
  that $\Ta_j - \Ta_{j-1} \leq \Tiat$. However,
  Lemma~\ref{lem:initcon-assump} guarantees that if $\Ta_j = \tau_a $
  is feasible then so is $\Ta_j = \tau_b$ for any
  $\tau_b \geq \tau_a$.  Using this for the case $\tau_a = \Ta_{j}$
  and $\tau_b = \tau_j$, we would deduce that $\Ta_j = \tau_j$ is
  feasible, which is a contradiction.  The rest of the proof is the
  same as in the first case of the proof of~(a).
\end{proof}

Note that in~\eqref{eq:Tiat}, $\Tiat$ is defined as the solution of a
maximization problem. However, since the maximization problem involves
only the parameters of the system, it could be solved offline. In
fact, we can give an analytical expression for $\Tiat$, which we
present in the next result.

\begin{corollary}\longthmtitle{Analytical expression for $\Tiat$}
  \begin{equation}\label{eq:Tiat-exp}
    \Tiat =  \begin{cases}
      \sigma_0 \Tnom, &\text{if } \ulinevv > \nunom ,
      \\
      \max\{ \sigma_0 \Tnom, \Tfol(\ulinevv) \}, &\text{if } \ulinevv
      \leq \nunom ,
    \end{cases}
  \end{equation}
  where
  \begin{align*}
    \Tfol(v) & \triangleq \frac{ (\nunom)^2 - v^2 }{ 2 u_M v^M } +
    \frac{ \sigma_0 \Dc( v, v^M ) }{ v^M } + \frac{ \nunom - v }{ u_M
    } .
  \end{align*}
\end{corollary}
\begin{proof}
  The case of $\ulinevv > \nunom$ follows directly from the fact that
  maximization $\Lc$ in~\eqref{eq:Tiat} is infeasible. Thus, now we
  assume $\ulinevv \leq \nunom$. By direct computation, we see that
  \begin{align*}
    \frac{ \partial \Lc(d,v) }{ \partial d }
    \begin{cases}
      < 0, \ & 2 u_M d < (v^M)^2 - v^2 \\
      = 0, \ & 2u_M d \geq (v^M)^2 - v^2 .
    \end{cases}
  \end{align*}
  Thus, it follows that
  \begin{align*}
    \max_{ \substack{d \geq \frac{ (\nunom)^2 - v^2 }{ 2u_M }, \\ v
        \in [\ulinevv, \nunom]} } \Lc(d,v) &= \!\! \max_{ \substack{d
        = \frac{ (\nunom)^2 - v^2 }{ 2u_M }, \\ v \in [\ulinevv,
        \nunom]} } \Lc(d,v) = \!\! \max_{ v \in [\ulinevv, \nunom]}
    \Tfol(v) \\
    &= \Tfol(\ulinevv) ,
  \end{align*}
  where the final equality follows from the fact that $\Tfol$ is a
  decreasing function of $v$.
\end{proof}

The next result summarizes the guarantees provided by the
\localvehcontrol controller~\eqref{eq:vehicle-control} regarding the
satisfaction of the constraints on safety and approach times.

\begin{theorem}\longthmtitle{Provably safe sub-optimal distributed
    control under finite-time constraints}\label{thm:sf-control}
  Consider a string of vehicles $\until{N}$ whose dynamics are
  described by~\eqref{eq:vehicle-dyn} under the \localvehcontrol
  controller~\eqref{eq:vehicle-control}.  Assume that $\xv_1(0) \leq {
    (v^M)^2 }/{ 2u_m } - { (\nunom)^2 }/{ 2u_M }$, that $\tau_j \in
  [\taue_j, \infty)$ and that the vehicles are in a safe configuration
  initially, ($\sigma_j(0) \geq 1$ for all $j \in \{2, \ldots, N\}$).
  Then,
  \begin{enumerate}
  \item inter-vehicle safety is ensured for all vehicles and for all
    time subsequent to $0$ (i.e., $\sigma_j(t) \geq 1$ for all $j \in
    \{2, \ldots, N\}$ and $ t \ge 0$),
  \item the first vehicle approaches the target region at $\tau_1$,
    each vehicle travels with a velocity of at least $\nunom$ at the
    time of approaching the target region and subsequent to it and
  \item for each $j \in \{2, \ldots, N\}$, if $\tau_j - \Ta_{j-1} \leq
    \Tiat$ then $\Ta_j - \Ta_{j-1} \leq \Tiat$. Alternatively, if
    $\tau_j - \Ta_{j-1} \geq \Tiat$, then $\Ta_j = \tau_j$.
  \end{enumerate}
\end{theorem}
\begin{proof}
  (a) Note that for $\sigma_j \in [1, \sigma_0]$, if $\zeta_j \in
  \Cc_s$, then $\sigma_j$ either stays constant, in the case of
  Lemma~\ref{lem:coupling-set}(c), or increases, in the case of
  Lemma~\ref{lem:coupling-set}(d). If on the other hand $\zeta_j
  \notin \Cc_s$, then it means $\vv_j < \vv_{j-1}$ and $\xv_{j-1} -
  \xv_j$ increases while $\Dc(\vv_{j-1},\vv_j)$ stays constant at $L$
  and thus $\sigma_j$ increases. Thus $\sigma_j(t) \geq 1$ is
  guaranteed for all vehicles $j \in \{2, \ldots, N\}$ and for all $t
  \geq 0$.

  (b) Since there is no vehicle in front of vehicle $1$, $u_1 = \guc$
  for all $t$. Initial feasibility then guarantees that $\Ta_1 =
  \tau_1$ and $\vv_1(\Ta_1) \geq \nunom$.

  Claim (c) follows directly from Proposition~\ref{prop:int-app-times}
  and by using induction.
\end{proof}

Note that we have not guaranteed optimality of our proposed solution
and in general it is only suboptimal. However, the uncoupled optimal
control mode ensures that the overall distributed controller is
optimum seeking for each individual vehicle.

\subsection{Integration with intelligent intersection management}

Here we elaborate on the application to intelligent intersection
traffic management, cf. Remark~\ref{rem:intersection-management}, of
our distributed control design for a string of vehicles under finite
time constraints.  We envision a system where each vehicle or groups
of vehicles communicate their aggregate information to a central
intersection manager. The intersection manager seeks to optimize the
schedule of the usage of the intersection by the vehicles. With the
information received, the manager schedules an intersection occupancy
time interval to each group of vehicles. The vehicles belonging to
each group then apply the \localvehcontrol
controller~\eqref{eq:vehicle-control} in order to satisfy the
prescribed schedule while also maintaining safety.  The aggregate
information required by the central intersection manager from each
group of vehicles has two pieces: constraints on the approach time
$\tau_1$ of the first vehicle in the group and a bound on the
occupancy time $\bartauocc \geq \Tx_N - \tau_1$ of the intersection
that could be guaranteed by the \localvehcontrol
controller~\eqref{eq:vehicle-control}. We discuss next how to compute
each element.

\subsubsection{Constraints on approach time of the first vehicle}
The constraints on $\tau_1$ could be computed by ignoring other
vehicles in the group, as in
Section~\ref{sec:relax-feas-app-times}. However, in doing so, ignoring
the initial conditions of the other vehicles in the group poses the
risk of lengthening the guaranteed upper bound $\bartauocc$ on the
occupancy time. The reasoning for this is better explained in terms of
earliest times of approach at the intersection of the vehicles.  If
$\taue_{j}$ for some $j > 1$ is significantly greater than
$\taue_{1}$, then having the vehicle $1$ slow down to approach the
intersection at a time later than $\taue_{1}$ will allow the string of
vehicles to meet a smaller guaranteed upper bound $\bartauocc_i$ on
the occupancy time.

Given this observation, we propose the following alternative way of
computing the constraints on the approach time of the first vehicle.
Recalling the interpretation of $\Tnom$ as the nominal inter-vehicle
approach time of vehicles in the group, we see that the earliest time
of approach for vehicle $j$ puts a constraint on the earliest time of
approach of the group, i.e., vehicle $1$, to be no less than
$\taue_{j} - (j-1) \Ac \Tnom$, where $\Ac \in [0, 1]$ is a design
parameter that determines the aggressiveness with which the prescribed
approach times for the vehicles are spaced. The smaller the value of
$\Ac$, smaller is the gap between the prescribed approach times and
hence greater is the aggressiveness with which the vehicles are forced
to enter the coupling set and use safe-following control. Hence, we
define the \emph{earliest time $\Te_1$ of approach} for the group of
vehicles as
\begin{equation}\label{eq:taum}
  \Te_1 \triangleq \max \{ \taue_{j} - (j-1) \Ac \Tnom :
  j \in \{1, \ldots, N\} \} .
\end{equation}
Similarly, we can also compute the \emph{latest time of approach}
$\Tl_1$ for the group of vehicles. Note that, under the assumptions
of Lemma~\ref{lem:initcon-assump}, $\Tl_1 = \infty$. Further, for
each vehicle $j \in \{1, \ldots, N\}$ in the group, we have the
intersection manager prescribe
\begin{equation}\label{eq:tauj}
  \tau_{j} \triangleq \tau_1 + (j - 1) \Ac \Tnom ,
\end{equation}
so that the only variable it must compute is~$\tau_1$. Thus, we see
that the parameter $\Ac$ influences the aggressiveness with which the
vehicles are driven into the safe-following mode. For example, $\Ac=0$
means that $\tau_j = \tau_1$ for all $j$, which necessarily means that
the each vehicle must enter the safe-following mode at least once. For
higher values of $\Ac$, there is a greater chance of the uncoupled
controller $\guc$ for vehicle $j$ being feasible until its approach
time $\Ta_j$. Given the constraints that the scheduler takes into
account, we have $\tau_1 \in [ \Te_1, \Tl_1 ]$. This, together
with~\eqref{eq:taum}, implies that $\tau_{j} \in [ \taue_{j},
\taul_{j} ]$, i.e., the sequence $\{\tau_j\}_{j=1}^N$ of approach
times prescribed by the intersection manager is feasible ignoring the
safety constraints.

\subsubsection{Guaranteed bound on occupancy time of intersection}

Given the sequence of approach times prescribed in~\eqref{eq:tauj} by
the intersection manager, the following result builds on
Proposition~\ref{prop:int-app-times} to provide a guaranteed upper
bound on the occupancy time of the target region by the group of
vehicles.

\begin{corollary}\longthmtitle{Guaranteed upper bound on occupancy
    time of the group of vehicles}\label{cor:occ-time}
  For the string of vehicles $\until{N}$, suppose $\tau_1 \geq
  \taue_1$, where $\taue_1$ is given by~\eqref{eq:taum}, and $\tau_j$
  for $j \in \ntil{N}$ satisfies~\eqref{eq:tauj}. Then, the occupancy
  time~$\tauocc \triangleq \Tx_N - \Ta_1$ is upper bounded as
  $\tauocc \le \bartauocc$, where
  \begin{align}\label{eq:bartauocc}
    \bartauocc = (N - 1) \Tiat + \max \left\{ \frac{ L + \Delta }{
        \nunom } , \Tiat \right\} .
  \end{align}
\end{corollary}
\begin{proof}
  From Theorem~\ref{thm:sf-control}, we know that $\Ta_1 = \tau_1$. We
  also know that $\Ta_j \geq \tau_j$ for each $j \in \ntil{N}$. Thus,
  as a result of~\eqref{eq:tauj}, we know that $\tau_j - \Ta_{j-1}
  \leq \Tnom < \Tiat$ for all $j \in \ntil{N}$. Hence, from
  Proposition~\ref{prop:int-app-times}, we see that the last vehicle
  $N$ approaches the target region at time $\Ta_{N}$ satisfying
  $\Ta_{N} \leq \Ta_{1} + (N - 1) \Tiat$. Since each vehicle travels
  with a velocity of at least $\nunom$ after approaching the
  intersection, the vehicle $N$ (and thus the group of vehicles) exits
  the intersection no later than $ \Ta_{N} + \frac{ L + \Delta }{
    \nunom }$. That is,
  \begin{equation*}
    \Tx_N \leq \Ta_{N} + \frac{ L + \Delta }{ \nunom } \leq \Ta_{1} +
    (N - 1) \Tiat + \frac{ L + \Delta }{ \nunom } ,
  \end{equation*}
  from which the result follows.
\end{proof}

The reasoning for the inclusion of $\Tiat$ in the second term
of~\eqref{eq:bartauocc} is as follows. There may be a second group of
vehicles that uses the intersection immediately after the first
group. Thus, we would like to have a safe-following distance between
the last vehicle of the first group and the first vehicle of the
second group even as it approaches the intersection at its assigned
time. The inclusion of the term $\Tiat$ ensures that if
$\tau_{N+1} \geq \tau_1 + \bartauocc$ then
$\tau_{N+1} - \Ta_{N} \geq \Tiat$, where $N+1$ is the index of the
first car in the second group of vehicles. Then, from
Proposition~\ref{prop:int-app-times}(b), it follows that
$\Ta_{N+1} = \tau_{N+1}$. This helps in scheduling the intersection
usage by several groups of vehicles with just the aggregate data of
$\tau_1$ and $\bartauocc$ for each group.

\section{Simulations}\label{sec:sim}

This section presents simulations of the vehicle string evolution
under the proposed \localvehcontrol
controller. Table~\ref{tab:sys-par} specifies the system parameters
employed in the simulations ($\Tnom$ and $\Tiat$ are computed
according to~\eqref{eq:Tnom} and~\eqref{eq:Tiat}, while the remaining
parameters are design choices or are typical of cars and arterial
roads ). All the units are given in SI units. For better intuition,
$v^M$ and $\nunom$ are equivalently $60$km/h and $48$km/h,
respectively.
\begin{table}[!htb]
  \caption{System parameters}\label{tab:sys-par}
  \centering
  \begin{tabular}{l l l}
    Parameter & Symbol & Value \\
    \hline
    Car length & $L$ & $4$m \\
    Target region length & $\Delta$ & $12$m \\
    Max. speed limit & $v^M$ & $16.667$m/s\\
    Max. accel. & $u_M$ & $3$m/s$^2$ \\
    Min. accel. & $u_m$ &  $-4$m/s$^2$ \\
    Nominal speed of crossing & $\nunom$ &  $13.333$m/s\\
    Parameter in~\eqref{eq:coupling-set} & $\sigma_0$ & 1.2 \\
    Nominal inter-vehicle approach time & $\Tnom$ & $\approx 1.24$s
    \\
    Upper bound on inter-vehicle approach time & $\Tiat$ &
                                                           $\approx1.59$s
  \end{tabular}
\end{table}
We present five sets of simulations, labeled Sim1 to Sim5. In all of
them, the number of vehicles is $N = 8$.  The initial conditions are
randomly generated so that initial safety, $\sigma_j(0) \geq 1$, is
satisfied for all $j \in \ntil{N}$ and~\eqref{eq:suff-cond} holds for
all $j \in \until{N}$. In all simulations but Sim5, the initial
conditions are the same, with the only distinguishing factor being how
the prescribed approach times $\tau_j$ are determined.

In Sim1, shown in Figure~\ref{fig:sim1}, the prescribed approach times
are randomly generated, with the only constraints being
$\tau_j \geq \taue_j$.
\begin{figure}[!htpb]
  \centering
  \subfigure[\label{fig:T1}]
  {\includegraphics[width=0.23\textwidth]{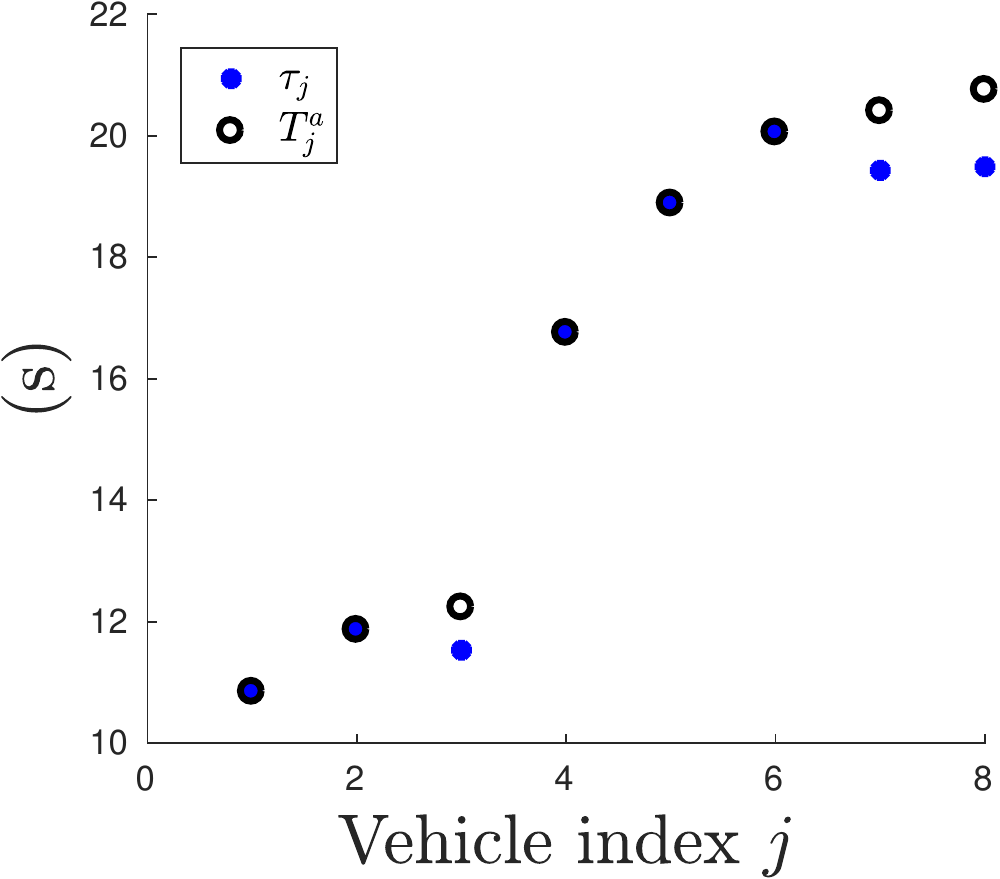}}
  \subfigure[\label{fig:X1}]
  {\includegraphics[width=0.23\textwidth]{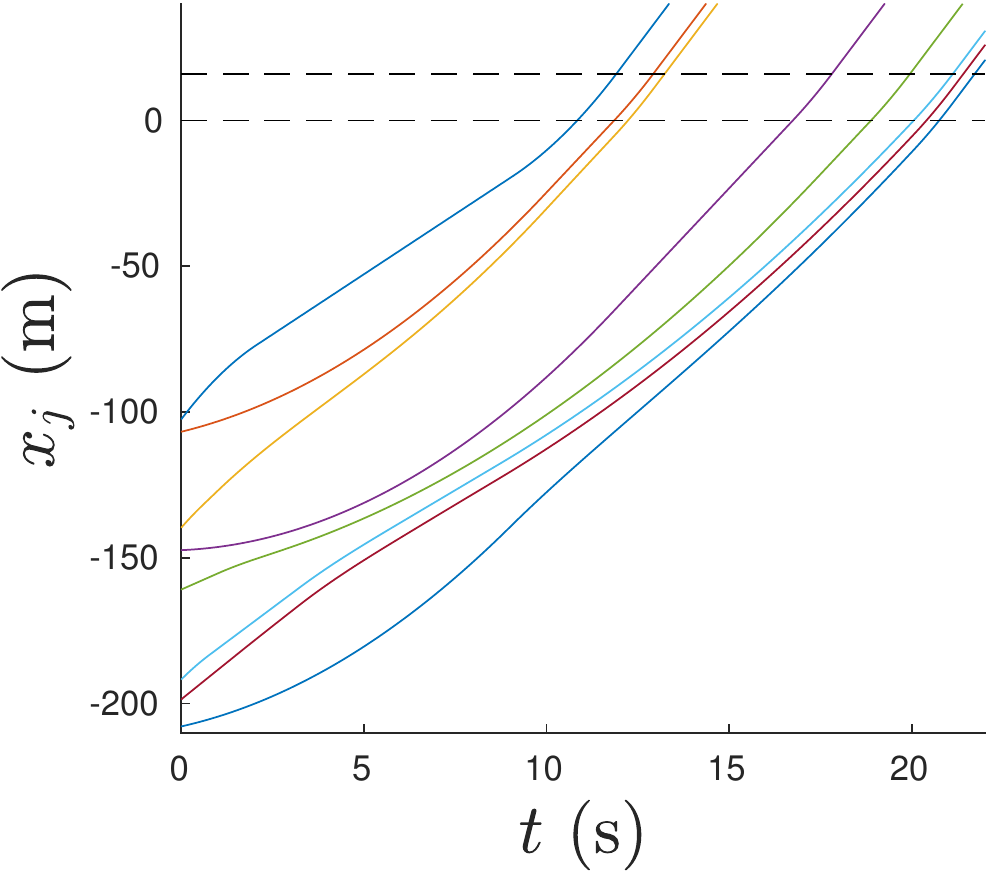}}
  \\
  \subfigure[\label{fig:S1}]
  {\includegraphics[width=0.23\textwidth]{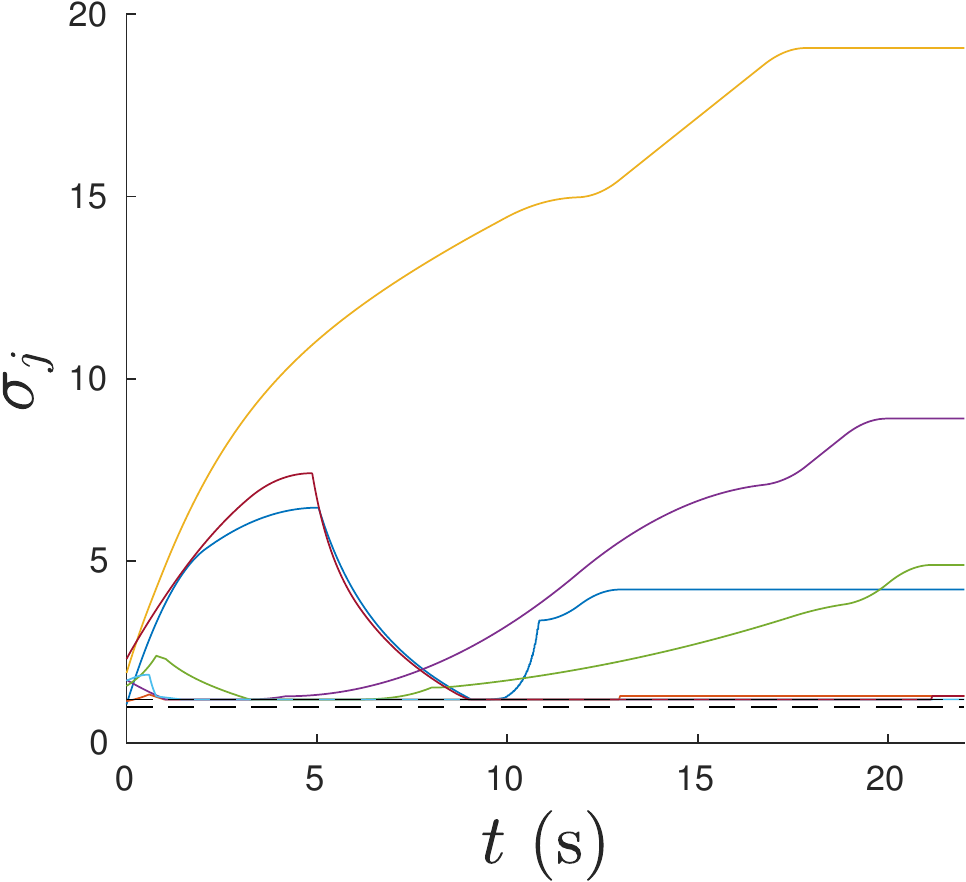}}
  \subfigure[\label{fig:V1}]
  {\includegraphics[width=0.23\textwidth]{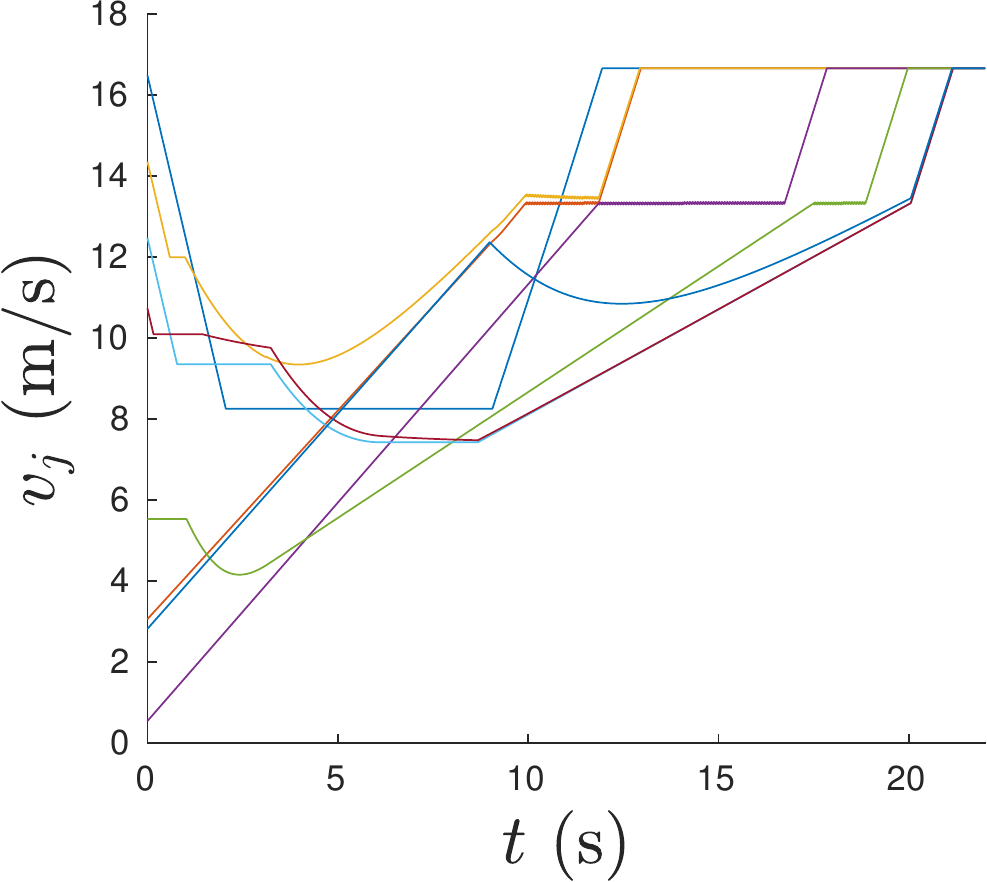}}
  \vspace*{-1ex}
  \caption{Results for Sim1. (a) Prescribed and actual approach times
    of the vehicles in the string. (b) Evolution of the position of
    the vehicles. The region between the dotted lines is the target
    region. (c) Evolution of the safety ratios. The dotted lines are
    at $\sigma = 1$ and $\sigma_0 = 1.2$. (d) Evolution of the
    velocities.}\label{fig:sim1}
\end{figure}
The choice of random $\tau_j$ does not in general result in cohesion
of the vehicles as they pass through the target region, as can be
clearly seen in Figures~\ref{fig:X1}-\ref{fig:S1}. The occupancy time
in this case is $\tauocc = 10.88$s.  In this case, we do not have
either an analytical expression for the bound on the occupancy time
(which is why the prescribed approach times must instead be
constrained, for example as in~\eqref{eq:tauj}, in the context of
intersection management).

In Sim2 to Sim4, $\tau_1 = \Te_1$ is chosen with $\Te_1$ as
in~\eqref{eq:taum} and the remaining $\tau_j$ are determined according
to~\eqref{eq:tauj}. In Sim2, shown in Figure~\ref{fig:sim2}, we choose
$\Ac = 1$.
\begin{figure}[!htpb]
  \centering
  \subfigure[\label{fig:T2}]
  {\includegraphics[width=0.23\textwidth]{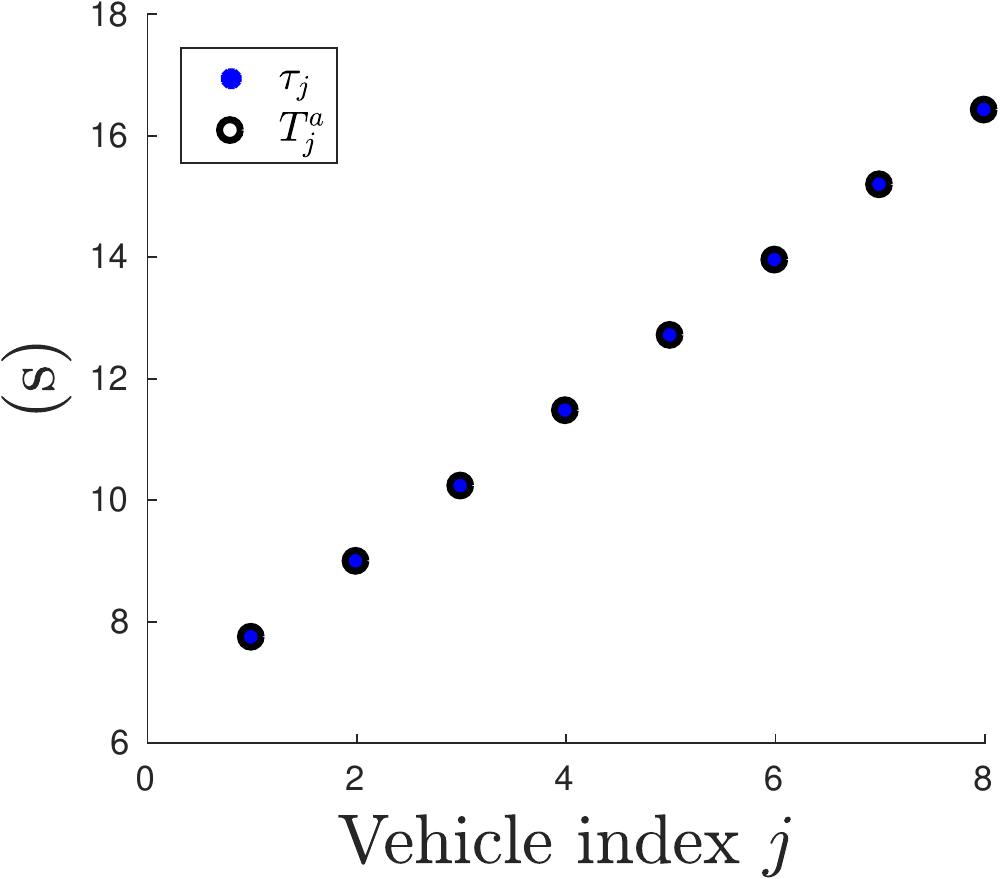}}
  \subfigure[\label{fig:X2}]
  {\includegraphics[width=0.23\textwidth]{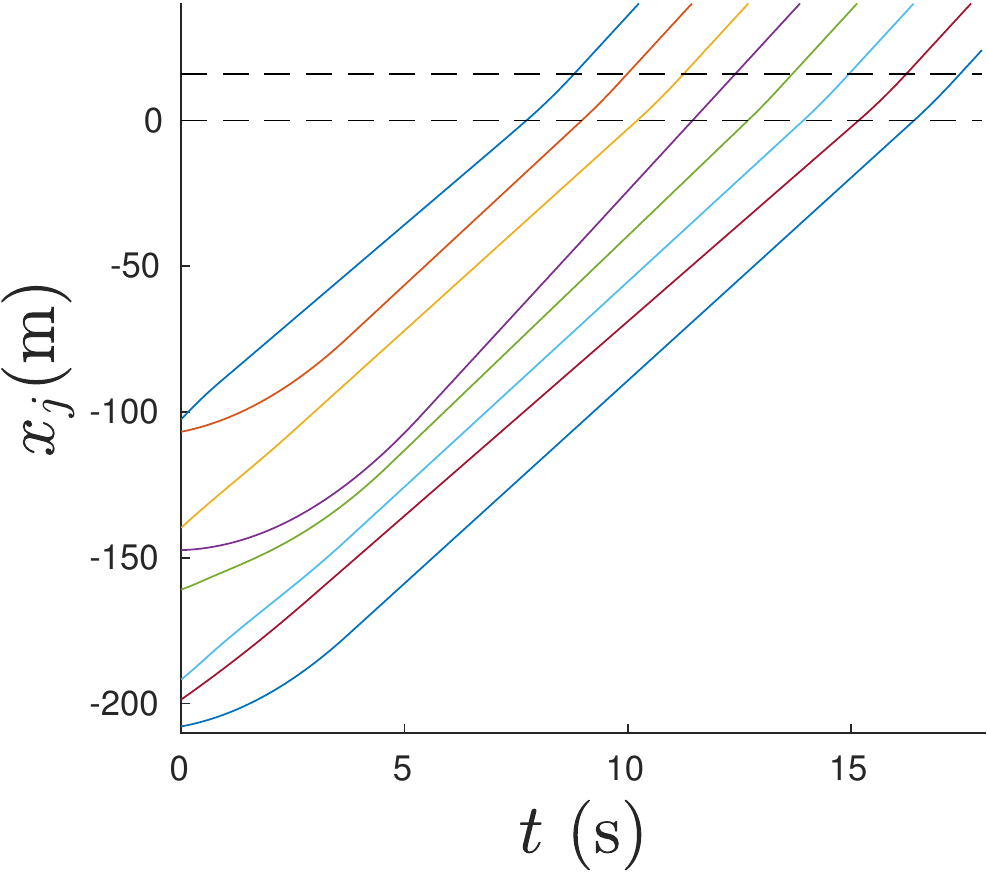}}
  \\
  \subfigure[\label{fig:S2}]
  {\includegraphics[width=0.23\textwidth]{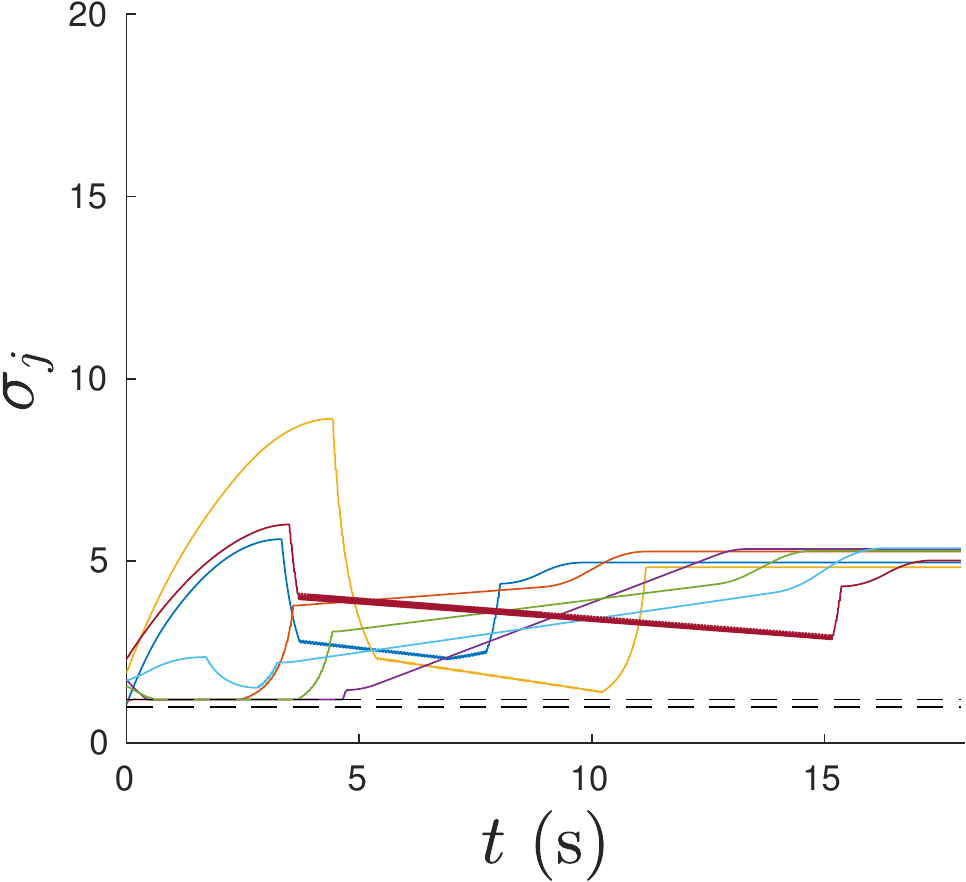}}
  \subfigure[\label{fig:V2}]
  {\includegraphics[width=0.23\textwidth]{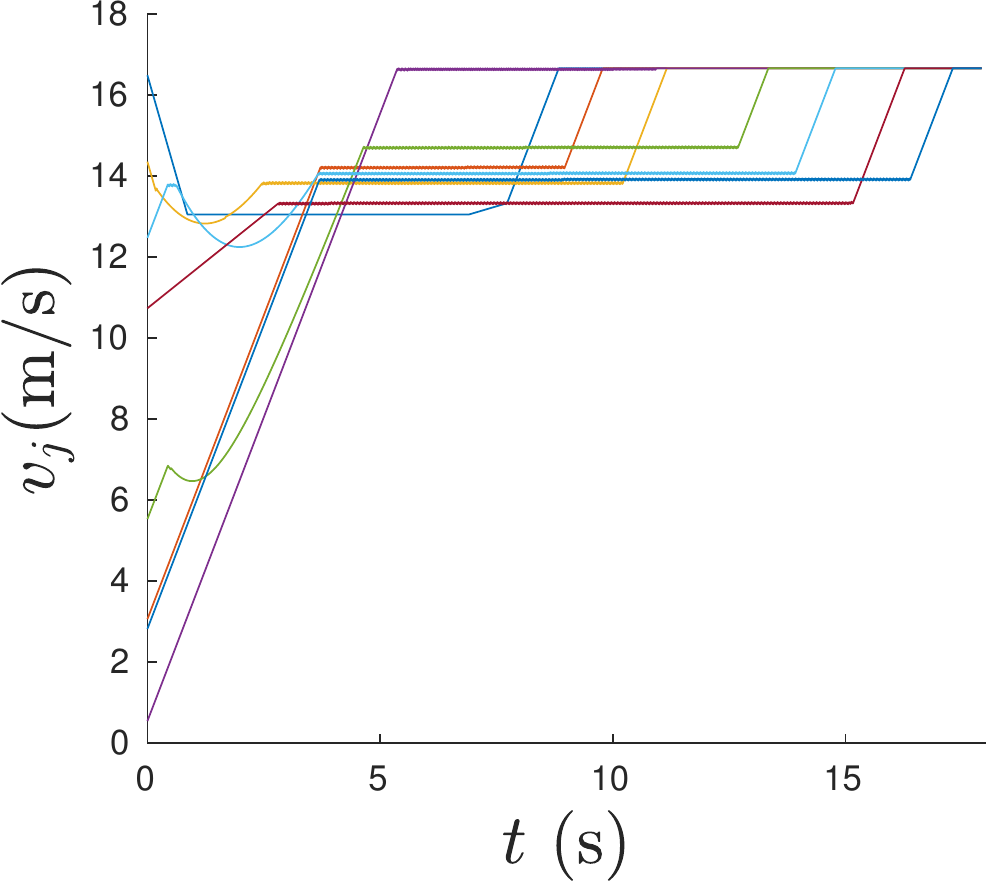}}
  \vspace*{-1ex}
  \caption{Results for Sim2. (a) Prescribed and actual approach times
    of the vehicles in the string. (b) Evolution of the position of
    the vehicles. The region between the dotted lines is the target
    region. (c) Evolution of the safety ratios. The dotted lines are
    at $\sigma = 1$ and $\sigma_0 = 1.2$. (d) Evolution of the
    velocities.}\label{fig:sim2}
\end{figure}
Figure~\ref{fig:T2} shows that the prescribed and the actual approach
times coincide for each vehicle. In fact, for $\sigma_0$ smaller than
around $4$, we have consistently observed that $\Ta_j = \tau_j$ for
each vehicle $j$. Figures~\ref{fig:X2}-\ref{fig:S2} demonstrate the
moderate cohesion that is achieved as vehicles cross the target
region. The occupancy time and the theoretical upper bound are
$\tauocc = 9.72$s and $\bartauocc = 12.72$s, respectively. While each
vehicle approaching the target region at its prescribed time is
desirable, the occupancy time $\tauocc$ is large.

In Sim3, shown in Figure~\ref{fig:sim3}, we demonstrate the utility of
the tuning parameter $\Ac$ in~\eqref{eq:taum}-\eqref{eq:tauj}. We choose
$\Ac = 0$, resulting in the prescribed approach times being all the
same, $\tau_j = \tau_1$ for all $j$. This necessarily forces vehicles
to interact through the coupling set and the safe-following controller
aggressively.  While the actual approach times are no longer equal to
their prescribed values (cf.  Figure~\ref{fig:T3}), this specification
results in high cohesion of vehicles as they cross the target region
(cf. Figures~\ref{fig:X3}-\ref{fig:S3}).
\begin{figure}[!htpb]
  \centering
  \subfigure[\label{fig:T3}]
  {\includegraphics[width=0.23\textwidth]{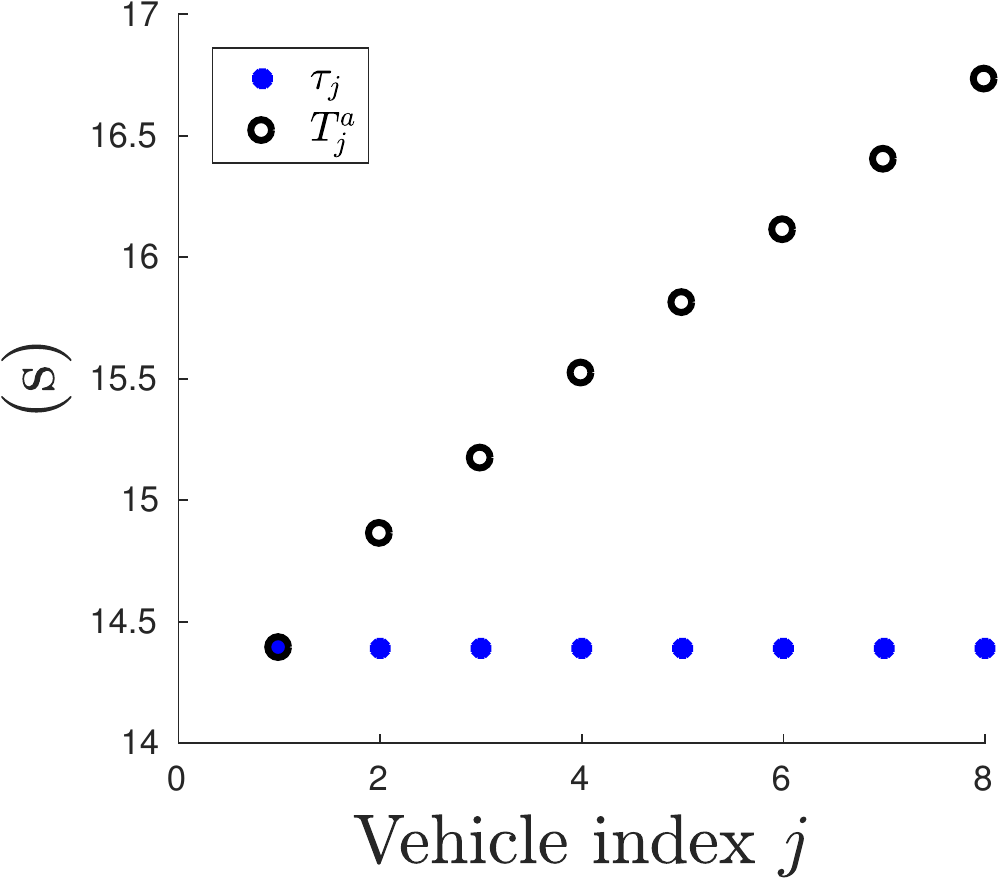}}
  \subfigure[\label{fig:X3}]
  {\includegraphics[width=0.23\textwidth]{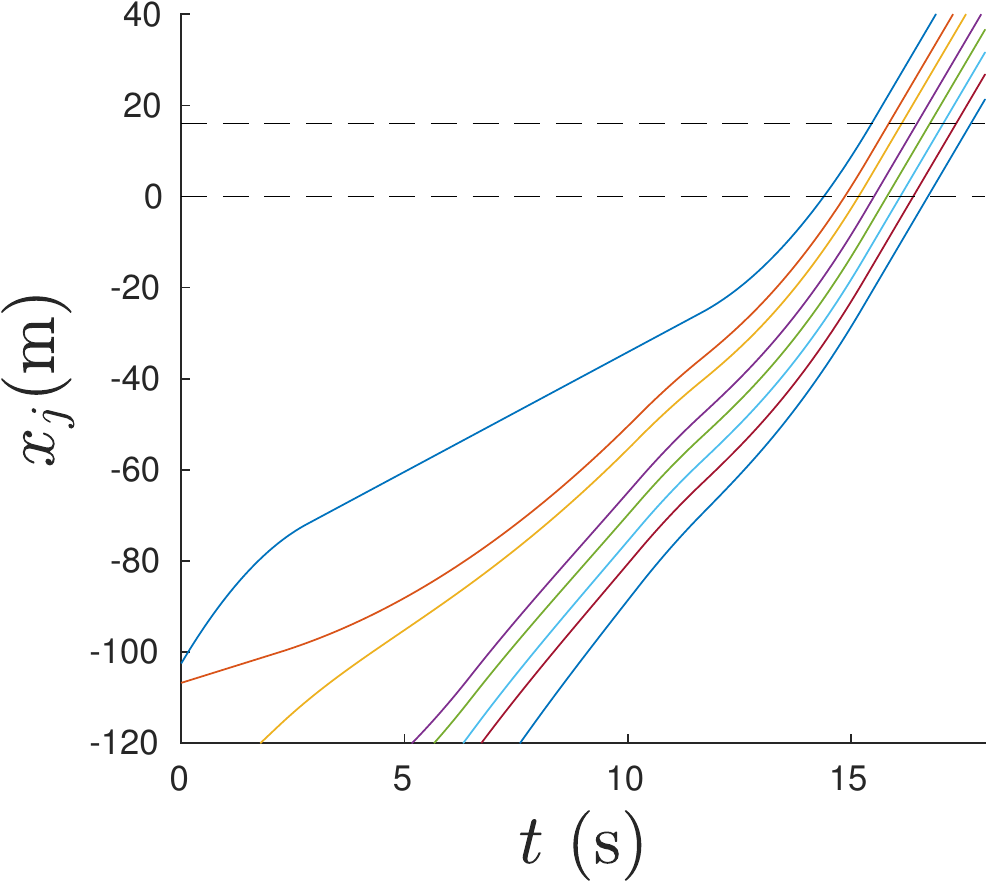}}
  \\
  \subfigure[\label{fig:S3}]
  {\includegraphics[width=0.23\textwidth]{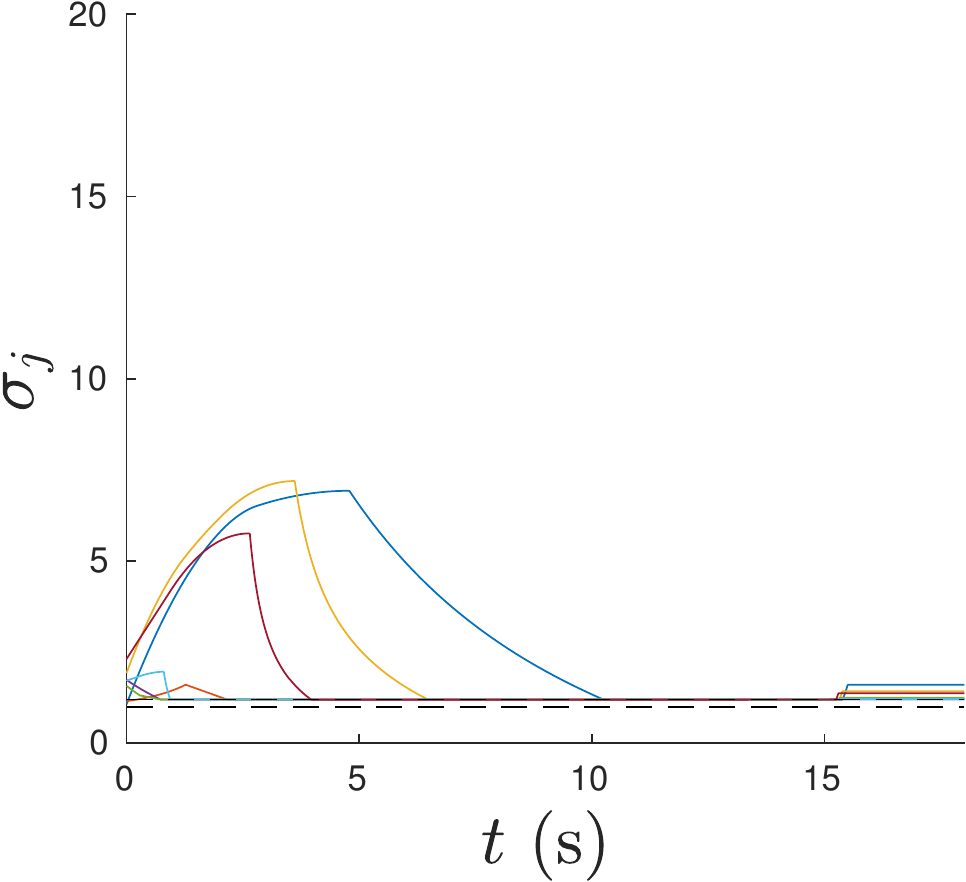}}
  \subfigure[\label{fig:V3}]
  {\includegraphics[width=0.23\textwidth]{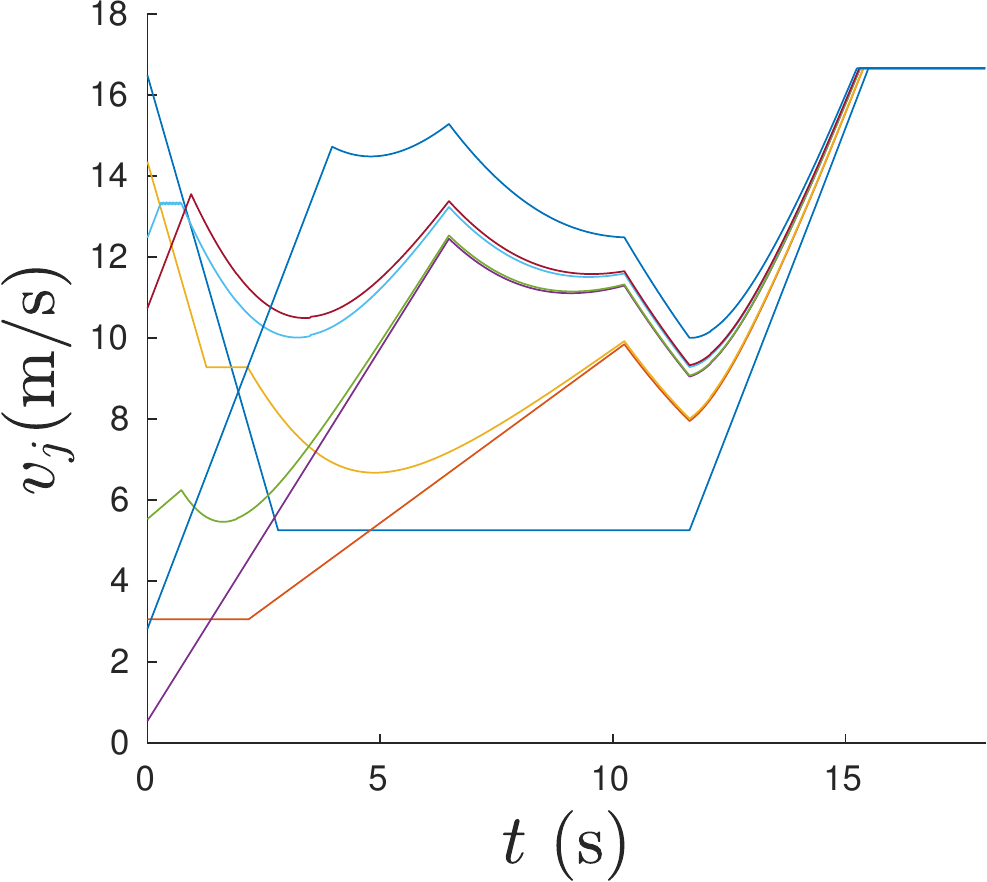}}
  \vspace*{-1ex}
  \caption{Results for Sim3. (a) Prescribed and actual approach
    times of the vehicles in the string. (b) Evolution of the position
    of the vehicles. The region between the dotted lines is the target
    region. (c) Evolution of the safety ratios. The dotted lines are
    at $\sigma = 1$ and $\sigma_0 = 1.2$. (d) Evolution of the
    velocities.}\label{fig:sim3}
\end{figure}
An important feature we have observed consistently in simulations with
$\Ac$ smaller than $1$ is the synchronization of the velocities
(cf. Figure~\ref{fig:V3}). By contrast, the velocity synchronization
in Figures~\ref{fig:V1} and~\ref{fig:V2} is an artifact of the
velocity saturation. We do not have analytical proof of this
phenomenon however. In Sim3, the occupancy time is much smaller at
$\tauocc = 3.3$s, although our bound on it still remains at
$\bartauocc = 12.72$s. Note also that the price of a smaller $\Ac$ is
a larger $T_1^e$, the earliest time that vehicle $1$ can approach the
intersection. This suggests an interesting trade-off between
intersection occupancy time and earliest time of
arrival. Figure~\ref{fig:sim3-UM} shows the control profile and the
evolution of the control mode of vehicle $8$ in Sim3. As expected, the
evolution of the control trajectory takes a complex form.
\begin{figure}[!htpb]
  \centering
  \subfigure[\label{fig:U3}]
  {\includegraphics[width=0.23\textwidth]{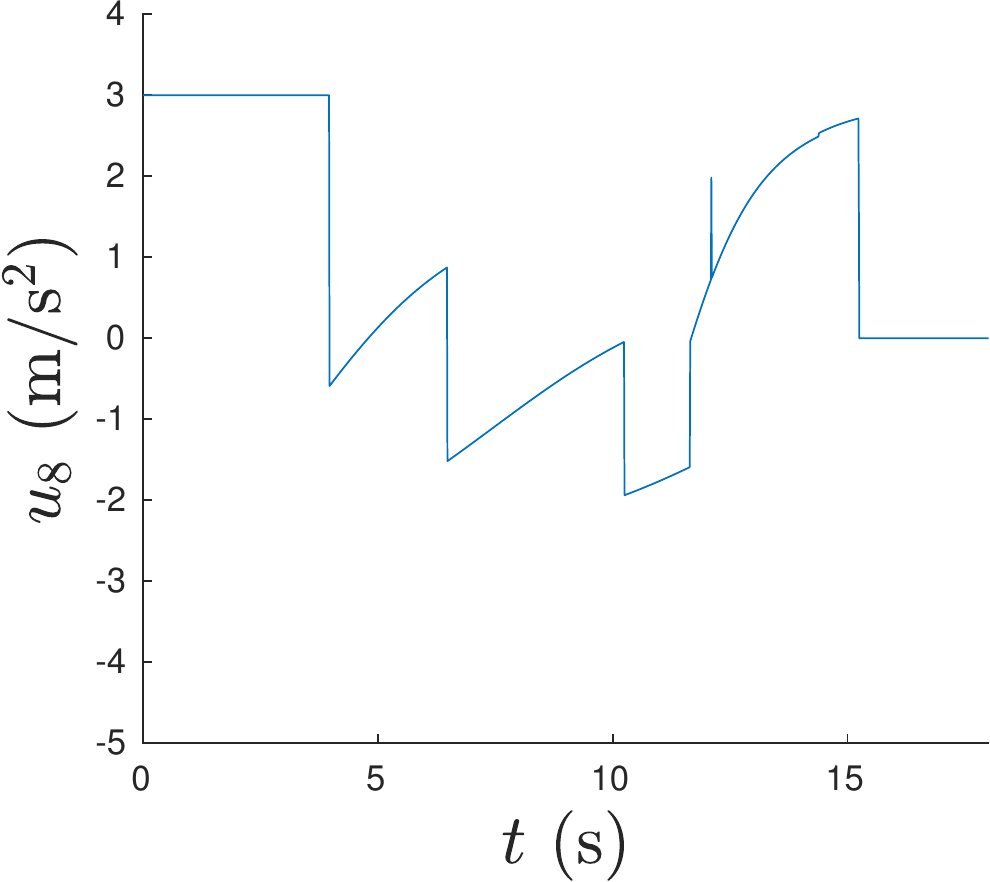}}
  \subfigure[\label{fig:M3}]
  {\includegraphics[width=0.23\textwidth]{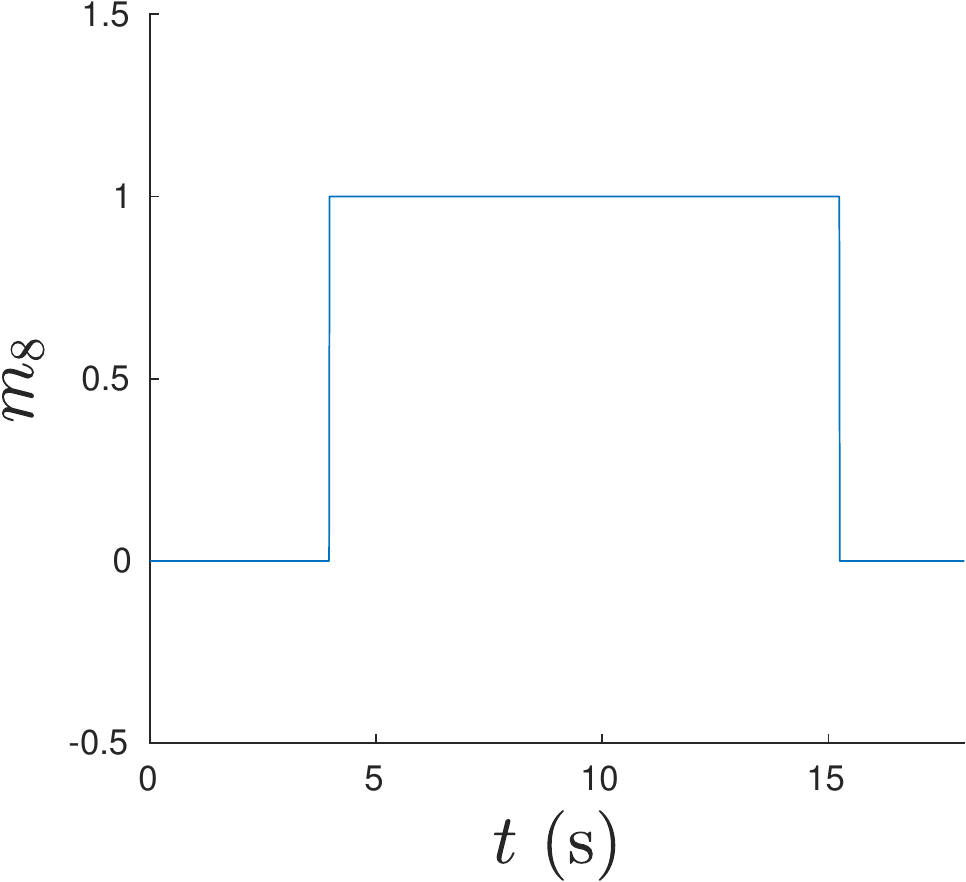}}
  \vspace*{-1ex}
  \caption{Results for Sim3. (a) Control profile for vehicle $8$. (b)
    Evolution of the control mode of vehicle $8$. $m_8 = 0$ and
    $m_8 = 1$ indicates vehicle $8$ is in the uncoupled mode and
    safe-following mode respectively.}\label{fig:sim3-UM}
\end{figure}

Finally, we illustrate in Figure~\ref{fig:Rp} the dependence of the
prescribed approach time of the first vehicle, the occupancy time, and
the time and fuel costs on the tuning parameter~$\Ac$.
\begin{figure}[!htpb]
  \centering
  \subfigure[\label{fig:Rp1}]
  {\includegraphics[width=0.23\textwidth]{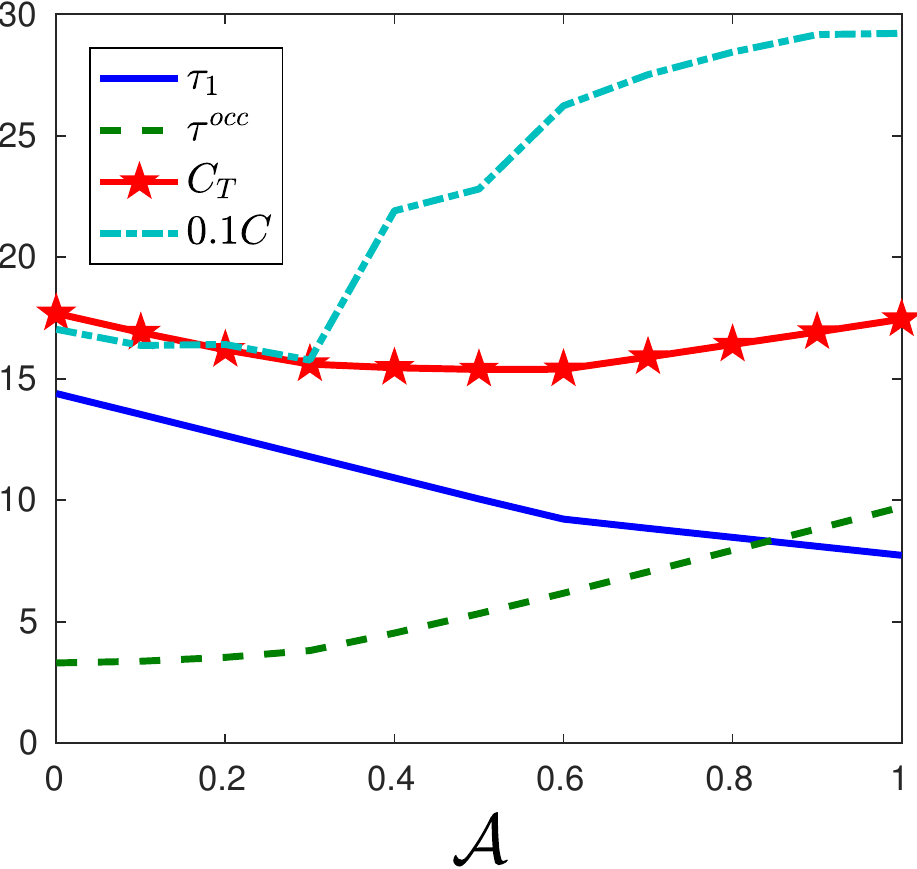}}
  \subfigure[\label{fig:Rp2}]
  {\includegraphics[width=0.23\textwidth]{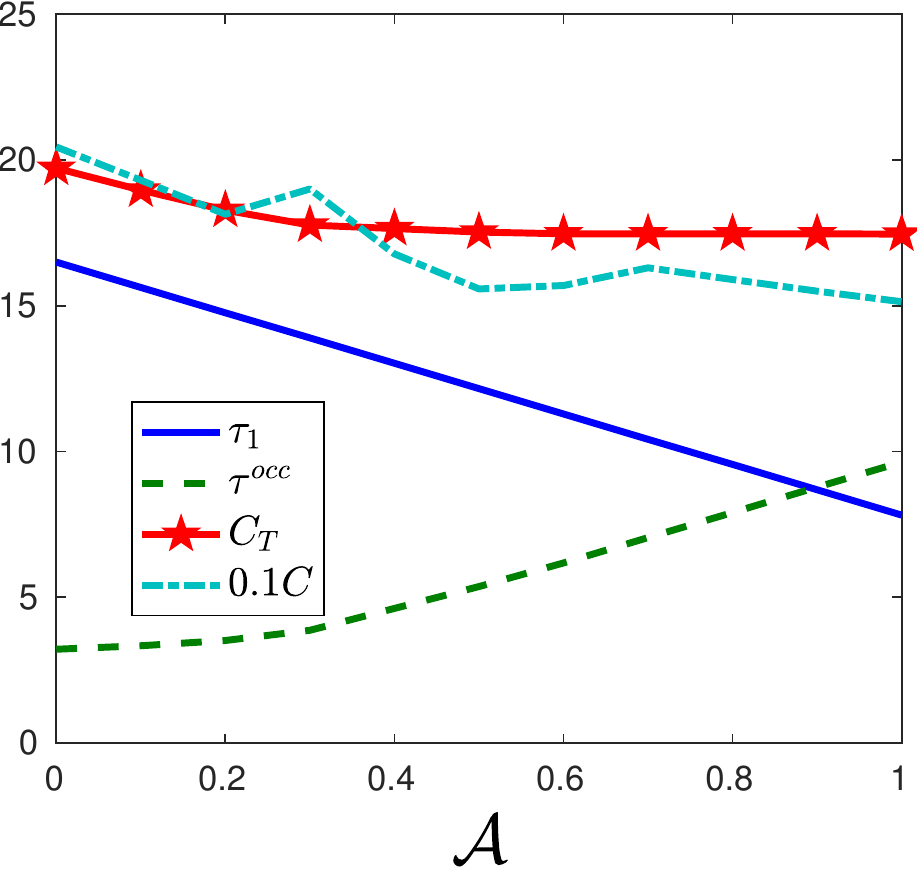}}
  \vspace*{-1ex}
  \caption{Results of Sim4 and Sim5. Each plot shows the dependence on
    the tuning parameter~$\Ac$ of the prescribed approach time of vehicle
    $1$, $\tau_1$, the occupancy time $\tauocc$, the time cost
    $C_T = \tau_1 + \tauocc$, and the fuel cost $C$.}\label{fig:Rp}
\end{figure}
In Sim4, the initial conditions are the same as in Sims 1-3, while in
Sim5 the initial conditions are different. The general trends for
$\tau_1$, $\tauocc$ and the time cost $C_T = \tau_1 + \tauocc$ are
independent of the initial conditions. However, we found the trend of
the fuel cost $C$ to be dependent on the initial conditions, as
illustrated in Figure~\ref{fig:Rp}.  Qualitatively, the earliest
approach time for the group of vehicles $T_1^e$, which is also set as
$\tau_1$ in the simulations, is decreasing with increasing $\Ac$,
while the occupancy time $\tauocc$ is increasing. An explanation for
the high dependence of the fuel cost $C$ on the initial conditions is
that these are the main factor determining the fraction of time that
vehicles spend in the coupled or uncoupled mode. Thus, the value of
the parameter $\Ac$ giving the best performance in terms of $C$
depends on the vehicles' initial conditions.

\section{Conclusions}\label{sec:conc}
We have studied the problem of optimally controlling a vehicular
string with safety requirements and finite-time specifications on the
approach time to a target region.  The main motivation for this
problem is intelligent management at traffic intersections with
networked vehicles.  We have proposed a distributed control
algorithmic solution which is provably safe (ensuring that even if
there was a communication failure, the vehicles could come to a
complete stop without collisions) and guarantees that the vehicles
satisfy the finite-time specifications under speed limits and
acceleration saturation.  We have also discussed how the proposed
distributed algorithm can be integrated into a larger framework for
intersection management for computer controlled and networked
vehicles. Finally, we have illustrated our results in simulation.
Future work will explore the derivation of tighter bounds on the
occupancy time of the intersection, optimizing the trade-off between
arrival time of the vehicle string at the intersection and occupancy
time, obtaining bounds on the overall fuel cost of the string,
refining the design of the safe-following controller, characterizing
how the design parameters affect the optimality of our design,
exploring robustness guarantees against vehicle addition and removal,
and incorporating privacy requirements.

\section*{Acknowledgments}
The research was supported by NSF Award CNS-1446891 and AFOSR Award
FA9550-10-1-0499.

\bibliographystyle{ieeetr} %
\bibliography{alias,FB,JC,Main,Main-add}

\begin{IEEEbiography}[{\includegraphics[width=1in,
    height=1.25in,clip,keepaspectratio]
    {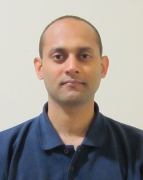}}]{Pavankumar Tallapragada}
  received the B.E. degree in instrumentation engineering from Shri
  Guru Gobind Singhji Institute of Engineering and Technology, Nanded,
  India, in 2005, the M.Sc. (Engg.)  degree in instrumentation from
  the Indian Institute of Science, Bangalore, India, in 2007, and the
  Ph.D. degree in mechanical engineering from the University of
  Maryland, College Park, MD, USA, in 2013. From 2014 to 2017, he was
  a Postdoctoral Scholar in the Department of Mechanical and Aerospace
  Engineering, University of California, San Diego, CA, USA. He is
  currently an Assistant Professor in the Department of Electrical
  Engineering, Indian Institute of Science. His research interests
  include event-triggered control, networked control systems,
  distributed control and networked transportation, and traffic
  systems.
\end{IEEEbiography}

\begin{IEEEbiography}[{\includegraphics[width=1in,
  height=1.25in,clip,keepaspectratio]{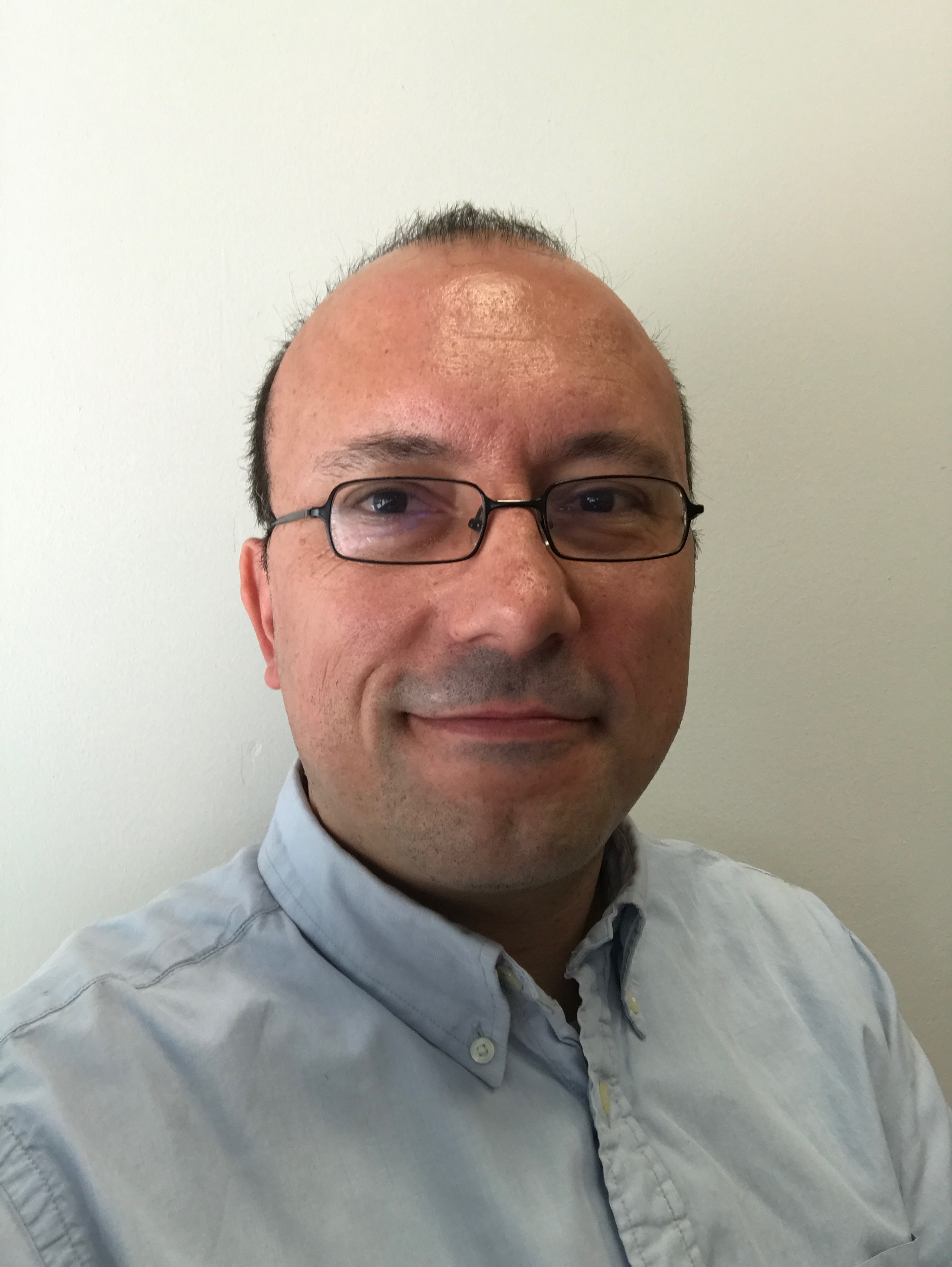}}]{Jorge
    Cort\'es}
  received the Licenciatura degree in mathematics from Universidad de
  Zaragoza, Zaragoza, Spain, in 1997, and the Ph.D. degree in
  engineering mathematics from Universidad Carlos III de Madrid,
  Madrid, Spain, in 2001.  He held post-doctoral positions with the
  University of Twente, Twente, The Netherlands, and the University of
  Illinois at Urbana-Champaign, Urbana, IL, USA. He was an Assistant
  Professor with the Department of Applied Mathematics and Statistics,
  University of California, Santa Cruz, CA, USA, from 2004 to 2007. He
  is currently a Professor in the Department of Mechanical and
  Aerospace Engineering, University of California, San Diego, CA,
  USA. He is the author of Geometric, Control and Numerical Aspects of
  Nonholonomic Systems (Springer-Verlag, 2002) and co-author (together
  with F. Bullo and S. Mart\'inez) of Distributed Control of Robotic
  Networks (Princeton University Press, 2009).  He has been an IEEE
  Control Systems Society Distinguished Lecturer (2010-2014) and is an
  IEEE Fellow.  His current research interests include distributed
  control, complex networks, opportunistic state-triggered control and
  coordination, distributed decision making in power networks,
  robotics, and transportation, and distributed optimization,
  learning, and games.
\end{IEEEbiography}

\end{document}